\documentclass[a4paper, 11pt]{article}

\usepackage[english]{babel}
\usepackage[utf8x]{inputenc}
\usepackage[T1]{fontenc}
\usepackage{amsmath}
\usepackage{amssymb}
\usepackage{amsthm}
\usepackage{stmaryrd}
\usepackage[pagewise, displaymath, mathlines]{lineno}%\linenumbers

\usepackage[a4paper,top=3cm,bottom=2cm,left=3cm,right=3cm,marginparwidth=1.75cm]{geometry}

\usepackage{enumerate}
\usepackage{amsmath}
\usepackage{graphicx}
\usepackage[colorinlistoftodos]{todonotes}
\usepackage[colorlinks=true, allcolors=blue]{hyperref}
\usepackage{mathtools}

\newtheorem{thm}{Theorem}
\newtheorem{defi}{Definition}
\newtheorem{prop}{Proposition}
\newtheorem{coroll}{Corollary}

\newtheorem{rmk}{Remark}
\newtheorem{lemma}{Lemma}

\newtheorem{notation}{Notation}

\newcommand\blfootnote[1]{%
  \begingroup
  \renewcommand\thefootnote{}\footnote{#1}%
  \addtocounter{footnote}{-1}%
  \endgroup
}

\newcommand{\RR}{\mathbb{R}}
\newcommand{\pD}{\partial D}

\newcommand{\CD}{\mathcal{D}}
\newcommand{\vertiii}[1]{{\left\vert\kern-0.25ex\left\vert\kern-0.25ex\left\vert #1 
    \right\vert\kern-0.25ex\right\vert\kern-0.25ex\right\vert}}
\newcommand{\bb}{\color{black}}
\newcommand{\dd}{\color{black}}

\title{A semigroup approach to the convergence rate of a collisionless gas.}
\author{Armand Bernou\footnote{ Sorbonne  Universit\'e,  CNRS,  Laboratoire  de  Probabilit\'e,  Statistique  et  Mod\'elisation,  F-75005 Paris, France. E-mail address: armand.bernou@sorbonne-universite.fr.} }

\begin{document}

\maketitle

\begin{abstract}
We study the rate of convergence to equilibrium for a collisionless (Knudsen) gas enclosed in a vessel in dimension $n \in \{2,3\}$. By semigroup arguments, we prove that in the $L^1$ norm, the polynomial rate of convergence $\frac{1}{(t+1)^{n-}}$ given \bb in \cite{TsujiRelaxationFreeMolecularGas2010}, \cite{Kuo2013} and \cite{Kuo2014} \dd can be extended to any $C^2$ domain, with standard assumptions on the initial data. This is to our knowledge, the first quantitative result in collisionless kinetic theory in dimension equal to or larger than 2 relying on deterministic arguments that does not require any symmetry of the domain, nor a monokinetic regime. The dependancy of the rate with respect to the initial distribution is detailed. \bb Our study includes the case where the temperature at the boundary varies. \dd
%We address for the first time the case where the temperature at the boundary varies and the equilibrium distribution is no longer explicit. 
The demonstrations are adapted from a deterministic version of a subgeometric Harris' theorem recently established by  Ca\~nizo and Mischler \cite{CanizoMischler}. We also compare our model with a free-transport equation with absorbing boundary.
\end{abstract}



\textbf{Acknowledgements:} I would like to thank my supervisor, St\'ephane Mischler, for  suggesting  me  to study this  problem  and  for  all  the  fruitful  discussions  and  advices  he offered me. I would also like to thank Kazuo Aoki for pointing out the references \cite{Kuo2014} and \cite{Sone07}, and the anonymous referee for useful comments. This work was supported by grants from R\'egion \^Ile de France.

\section*{Introduction}

\paragraph{Model.} In this paper, we study the kinetic free-transport equation with Maxwell boundary conditions inside a domain $D$ in $\RR^n$ having closure $\bar{D}$, with $n \in \{2,3\}$:
\begin{equation}
\label{Problem1}
\left\{ 
	\begin{aligned}
    \partial_t f + v \cdot \nabla_x f  &= 0, \hspace{5.35cm}  (t,x,v) \in \RR_+^* \times G, \\
 \gamma_- f (t,x,v) &= K \gamma_+f(t,x,v), \hspace{3.3cm} (t,x,v) \in \RR_+ \times \partial_- G, \\ 
 f|_{t=0}(x,v)  &= f_0(x,v), \hspace{4.3cm} (x,v) \in G, 
	\end{aligned}
\right.
\end{equation}
where we use the notations $G := D \times \RR^n$, and, denoting $n_x$ the unit \textbf{inward} normal vector at $x \in \pD$, 
\[
\begin{aligned}
\partial_+ G &:= \{(x,v) \in \pD \times \RR^n, v \cdot n_x < 0\}, \\
\partial_- G &:= \{(x,v) \in \pD \times \RR^n, -(v \cdot n_x) < 0\}.
\end{aligned}
\]
Given  a function $\phi$ on $(0,\infty) \times \bar{D} \times \RR^n$, $\gamma_{\pm} \phi$
denotes its trace on $(0, \infty) \times \partial_{\pm} G$, provided this object is well-defined. The boundary operator $K$ is defined, for all $(t,x,v) \in \RR_+ \times \partial_- G$ and for $\phi$ supported on $(0,\infty) \times \partial_+ G$ such that $\phi(t,x,\cdot)$ belongs to $L^1(\{v': v' \cdot n_x < 0\})$, by
\begin{equation}
\begin{aligned}
\label{EqDefK}
 K\phi(t,x,v) &= \alpha(x) M(x,v) \int_{\{v' \in \RR^n: v' \cdot n_x < 0\}} \phi(t,x,v') |v' \cdot n_x| dv' \\
 &\quad + (1 - \alpha(x)) \phi(t,x,v - 2(v \cdot n_x)n_x).
 \end{aligned}
 \end{equation}
In this paper, we consider the standard (and physically relevant) case of the Maxwellian distribution at the boundary $\pD$,

\begin{equation}
\begin{aligned}
\label{EqDefM}
 M(x,v) = \frac{\tilde{c}(x)}{(2\pi\theta(x))^{\frac{n}{2}}} e^{-\frac{\|v\|^2}{2\theta(x)}}, \qquad x \in \pD, v \in \RR^n, 
 \end{aligned}
 \end{equation}
 where, for all $x \in \pD$, for $z \in \pD$, 
\begin{equation}\begin{aligned}
\label{EqDefTildeC}
 \tilde{c}(x) = \Big( \int_{\{v \cdot n_z < 0\}} \frac{1}{(2\pi\theta(x))^\frac{n}{2}} e^{-\frac{\|v\|^2}{2\theta(x)}} |v \cdot n_z| dv \Big)^{-1}, 
 \end{aligned}\end{equation}
which is independent of the choice of $z$ since the integrand is radial, so that 
\begin{equation}\begin{aligned}
\label{Eq2tildec}
\int_{\{v \cdot n_x < 0\}} M(x,v) |v \cdot n_x| dv = 1. 
\end{aligned}\end{equation}
The parameter $\theta(x)$ corresponds physically to the temperature at the point $x \in \pD$ of the boundary wall considered at rest. 
\vspace{.3cm}

\paragraph{Physical motivations.} This problem models the evolution of a Knudsen (collisionless) gas enclosed in the vessel $D$. For such diluted gases, the Lebesgue measure of the set of collisions between particles is 0, hence the collision operator of the Boltzmann equation describing statistically the dynamic, as introduced by Maxwell \cite{Maxwell}, vanishes. Particles in $D$ move according to the free-transport dynamic until they meet with the boundary. They are reflected at the boundary $\pD$ in a diffuse or specular manner, corresponding to the two terms in the definition of $K$: at a point $x \in \pD$, a fraction $1-\alpha(x)$ of the gas particles is specularly reflected, i.e., if $v \in \RR^n$ is the initial velocity, the outgoing velocity is given by

\vspace{-.6cm}

\begin{equation}\begin{aligned}
\label{EqEtax}
\eta_x(v) := v - 2(v \cdot n_x)n_x. 
\end{aligned}\end{equation}
The remaining fraction $\alpha(x)$ is diffusively reflected (and thus thermalized). The latter corresponds, physically, to the case where the particle is adsorbed by the wall before being re-emited inside the domain according to a new velocity distribution defined through $M$. More details on the derivation of this boundary condition can for instance be find in the monograph of Cercignani, Illner and Pulvirenti \cite[Chapter 8]{Cercignani94}.   For this model, the distribution function of the gas, $f(t,x,v)$, representing the density of particles in position $x \in \bar{D}$ with velocity $v \in \RR^n$ at time $t \geq 0$, satisfies (\ref{Problem1}). 

\paragraph{Link with the Boltzmann equation and convergence rate for (\ref{Problem1}).}
We study the rate of convergence towards equilibrium of (\ref{Problem1}). Taking $\theta \equiv \Theta$ for some $\Theta > 0$ so that $M$ only depends on $v$, the existence of a steady state and the convergence towards it (at least in a restricted context) is known since the work of Arkeryd and Cercignani \cite{ArkerydCercignani1993}. This equilibrium is given by, assuming the initial data to be of mass 1,
\begin{equation}\begin{aligned}
\label{EqDefEquilibrium}
f_{\infty}(x,v) = \frac{ e^{-\frac{\|v\|^2}{2\Theta}}}{|D| (2 \pi \Theta)^{\frac{n}{2}}}, \qquad (x,v) \in G,
\end{aligned}\end{equation}
where $|D|$ denotes the volume of $D$. In the collisional case, for instance when one studies the space-homogeneous Boltzmann equation with the same boundary condition as in (\ref{Problem1}), the famous H-theorem of Boltzmann gives a starting point from which Boltzmann \cite{BoltzmannLecturesgastheory1995} gave plausible arguments for the convergence towards an equilibrium as time goes to infinity. Once this convergence is established, a key question is the rate at which it occurs. Physically, one would also like to obtain an explicit form for the constant playing a role in the convergence rate, to avoid unsignificant values as one can find when working with the Poincaré recurrence theorem. Regarding Boltzmann equation with Maxwell boundary condition (or diffuse boundary condition, i.e. with $\alpha \equiv 1$) and constant temperature, there are strong reasons to believe that the convergence occurs at an exponential rate, i.e., that there exist $\lambda$, $C > 0$ such that if $f_t$ denotes the solution at time $t > 0$, for all $t \geq 0$,
\[ \|f_t - f_{\infty}\|_{L^1} \leq C e^{-\lambda t}, \]
where $f_{\infty}$ is the Maxwellian corresponding to the equilibrium of the system. On this matter, see for instance Villani \cite[18.5]{VillaniHypoco}, where it is established that the convergence rate is equal to, or better, than $t^{-\infty}$ in some Sobolev norm assuming some strong regularity estimates. However those estimates may not hold true in a non-convex setting, see \cite{GuoRegularity} for a discussion on those issues in a general context. When the initial data is close to the equilibrium, Guo \cite[Theorem 4]{Guo2010} proved the exponential convergence towards equilibrium. 

\vspace{.5cm}
This (expected) dissipative property of the previously mentioned Boltzmann equation is a consequence of two factors: the interactions with the boundary wall and the collision operator. On the other hand the model corresponding to (\ref{Problem1}) only deals with the interactions with the boundary wall. This leads to several natural questions.
\begin{enumerate}[i)]
\item Can we still prove a convergence towards an equilibrium ? In particular in the case where $\theta$ is not constant ?
\item Is the rate at which this convergence occurs exponential, as expected for the Boltzmann equation with similar boundary conditions ? If not, can we characterize this rate in a precise manner ?
\item Can we compute the corresponding constants explicitely ?
\end{enumerate}

\paragraph{Well-posedness and qualitative convergence.}  Arkeryd and Cercignani \cite{ArkerydCercignani1993} established the well-posedness of (\ref{Problem1}) in the $L^1$ setting. This allows one to associate a semigroup $(S_t)_{t \geq 0}$ to the evolution equation, so that given an initial datum $f_0$, $S_t f_0$ is the solution of (\ref{Problem1}) at time $t \geq 0$. The decay property of the distance with respect to the equilibrium, and thus the answer to the three questions above, can then be read at the level of the associated semigroup. 
\bb The fact that the answer to i) is positive is physically intuitive and has been established qualitatively in the convex setting and in dimension 3  by Arkeryd and Nouri \cite{ArkerydBoltzmannasymptoticsdiffuse1997}. \dd In the general setting, the obtention of a rate to answer ii) will provide an \textit{a posteriori} answer for question i) as well.

\paragraph{Known results for question ii).} Question ii) was first addressed numerically by Tsuji, Aoki and Golse \cite{TsujiRelaxationFreeMolecularGas2010}. They gave strong arguments to support the intuition that the rate of convergence in the $L^1$ norm will no more be exponential in this case, but rather polynomial of order $\frac{1}{t^n}$, where $n$ is the dimension of the problem. The absence of a spectral gap for the sole free-transport operator is a natural reason to think that the exponential rate cannot be reached for this model. Later, Aoki and Golse \cite{Aoki201187} proved that the rate of convergence is better than $\frac1t$ in the $L^1$ distance, with an additional assumption of symmetry of the domain and of the initial data, by means of Feller's renewal theory. \bb Still with this symmetry assumption on the domain and in dimensions $1$ to $3$, the problem was studied \textit{via} probabilistic methods by Kuo, Liu and Tsai \cite{Kuo2013}, see also \cite{Kuo2014} for the case where $\theta$ varies, with the sole diffuse condition, which corresponds to $\alpha \equiv 1$ with our notations. \dd The key idea is that the symmetry of the domain allows one to consider the intervals in time between two rebounds of a particle as independent and identically distributed random variables, and to deduce a law of large numbers from which one can control the flux of the solution at the boundary in the $L^{\infty}$ norm. \bb Kuo \cite{Kuo2015} later extended this result with similar tools to the case of the Maxwell boundary condition, in dimension 2.
\dd  Finally let us mention that Mokhtar-Kharroubi and Seifert \cite{Mokhtar2017} recently obtained an explicit polynomial rate in slab geometry (dimension 1). Their proof relies on a quantified version of Ingham's tauberian theorem.

\paragraph{Hypothesis and main result.} \bb While the methods used in \cite{Aoki201187}, \cite{Kuo2013}, \cite{Kuo2014} and \cite{Kuo2015} are difficult to adapt to a nonsymmetric setting, it seems intuitive to expect that the rate of convergence will be of the same order without this assumption. \dd In this work, we give, using a slightly modified version of the subgeometric Doeblin-Harris theory of Ca\~nizo and Mischler \cite{CanizoMischler}, an answer to questions i) and ii) and a partial answer to question iii) in the larger context of $C^2$ domains.

\vspace{.3cm}

Let us introduce some assumptions and key notations and present our main result. The dimension $n$ belongs to $\{2,3\}$. We endow $\RR^n$ and $\RR$ with the Lebesgue measure. The symbols $dx, dv, \dots$ denote this measure. 
We assume that the domain (open, connected) $D \subset \RR^n$ is bounded and $C^2$ with closure $\bar{D}$, and that the map $x \to n_x$ can be extended to the whole set $\bar{D}$ as a $W^{1,\infty}$ map, where $W^{1,\infty}$ denotes the corresponding Sobolev space. For any $k \in \mathbb{N}^*$, we use the Euclidian norm in $\RR^k$. We write $d(D)$ for the diameter of $D$
\[ d(D) = \sup_{(x,y) \in D^2} \|x-y\|. \] On $\bar{D} \times \RR^n$, setting \[\partial_0 G := \{(x,v) \in \pD \times \RR^n, v \cdot n_x = 0\},\] we define the map $\sigma$ by:
\begin{equation}\begin{aligned}
\label{NotatSigma}
\sigma(x,v) = \left \{ \begin{array}{l}
		 \inf \{t > 0, x + tv \in \pD\}, \qquad (x,v) \in \partial_- G \cup G, \\
		 0, \hspace{4.5cm} (x,v) \in \partial_+ G \cup \partial_0 G,
		 \end{array}
		 \right.
\end{aligned}\end{equation}
which corresponds to the time of the first collision with the boundary for a particle in position $x$ with velocity $v$ at time $t = 0$. The $L^1$ space on $G$, denoted $L^1(G)$ is the space of measurable $\RR$-valued functions $f$ such that 
\[ \|f\|_{L^1} := \int_{G} |f(x,v)| dv dx < \infty. \]
For any non-negative measurable function $w$ defined on $G$, we introduce the weighted $L^1$ space $L^1_w(G) = \{f \in L^1(G), \|fw\|_{L^1} < \infty\}$ endowed with the norm $\|f\|_{w} := \|fw\|_{L^1}$. For any function $f \in L^1(G)$, we define the mean of $f$ by
\begin{equation}\begin{aligned}
\label{EqDefNotationMass}
\langle f \rangle := \int_{G} f(x,v) dv dx.
\end{aligned}\end{equation}
For the function $\alpha: \pD \to [0,1]$ playing a role in the boundary condition, we assume that
\bb there exists a constant $c_0 \in (0,1)$ such that
\begin{equation}\begin{aligned}
\label{EqHypoAlpha}
\alpha(x) \geq c_0, \quad \forall x \in \pD.
\end{aligned}\end{equation}
\dd
This condition implies that $1 - \alpha(x) \leq 1 - c_0$ for all $x \in \pD$, a fact that will allow us to control the contribution of the specular component of the reflection at the boundary.
\bb
\begin{rmk}
In the case of the diffuse boundary condition, that is, when $\alpha \equiv 1$, our condition implies that one has to choose some $c_0 \in (0,1)$ rather than the value $1$ itself. Any value $c_0 \in (0,1)$ will satisfy (\ref{EqHypoAlpha}) in this case.
\end{rmk} 
\dd

We define the constant $c_4 \in (0,1)$ by the equation 
\begin{equation}\begin{aligned}
\label{EqDefc4}
(1-c_4)^4 = (1-c_0),
\end{aligned}\end{equation}
so that, for $i \in [1,4]$,  \[(1-c_4)^i \geq (1-c_0) .\]
Finally, we assume that the temperature function $\theta: \pD \to \RR_+$ is continuous, positive on $\pD$ compact, so that $(x,v) \to M(x,v)$ is continuous and positive.
We introduce the weights $\omega_{i}$ for $i \in \{1,\dots,4\}$ defined by setting, for all $(x,v) \in \bar{D} \times \RR^n$,
\begin{equation}\begin{aligned}
\label{EqDefWeightsOmega}
 \omega_{i}(x,v) = \Big(e^2 + \frac{d(D)}{\|v\| c_4} - \sigma(x,-v)\Big)^{i} \ln \Big(e^2 + \frac{d(D)}{\|v\| c_4} - \sigma(x,-v) \Big)^{-1.6}.
\end{aligned}\end{equation}
Note that $2 > 1.6 > \frac{n+1}{n}$ for $n \in \{2,3\}$. The idea behind this choice is that we will be able to interpolate a first result for $\omega_{n+1}$ by considering the weight $\omega_{n+1}^{\frac{n}{n+1}}$, and that the exponent of the logarithmic factor will still be smaller than $-1$.
 Our main results are the following.
\begin{thm}
\label{ThmMain}
There exists a constant $C > 0$ such that for all $t \geq 0$, for all $f,g \in L^1_{\omega_{n+1}}(G)$ with $\langle f \rangle = \langle g \rangle$, there holds
\[ \|S_t (f-g)\|_{L^1} \leq \frac{C \ln(1+t)^{n+2}}{(1+t)^{n+1}} \|f-g\|_{\omega_{n+1}}. \]
\end{thm}

For $i \in \llbracket n-1, n+1 \rrbracket$, we define the weight $m_i$ by setting, for all $(x,v) \in \bar{D} \times \RR^n$,
\begin{equation}\begin{aligned}
\label{EqDefWn}
m_{i}(x,v) = \Big(e^2 + \frac{d(D)}{\|v\| c_4} - \sigma(x,-v)\Big)^{i} \ln \Big(e^2 + \frac{d(D)}{\|v\| c_4} - \sigma(x,-v) \Big)^{-1.6 \frac{n}{n+1}}. 
\end{aligned}\end{equation}

A second theorem, which relies on similar arguments, answers question i) above even in the case where $\theta$ is not constant.
\begin{thm}
\label{ThmEquilibrium}
There exists a unique $f_{\infty} \in L^1_{m_n}(G)$ such that $0 \leq f_{\infty}$, $\langle f_{\infty} \rangle = 1$ satisfying
\[ \begin{aligned} 
v \cdot \nabla_x f_{\infty}(x,v) &= 0, \hspace{1.25cm} \qquad \qquad  (x,v) \in G, \\
\gamma_- f_{\infty}(x,v) &= K \gamma_+ f_{\infty}(x,v), \qquad (x,v) \in \partial_- G.
\end{aligned} \]
\end{thm}
%Note that, as announced, $(-\frac{1.6  n}{n+1}) \in (-2,-1)$ for $n \in \{2,3\}$ and that $m_n = \omega_{n+1}^{\frac{n}{n+1}}$.
\color{black}
Regarding the convergence towards equilibrium, we obtain by interpolation the following corollary from Theorem \ref{ThmMain}.

\begin{coroll}
\label{MainCorol}
There exists a constant $C' > 0$ such that for all $t \geq 0$, for all $f \in L^1_{m_n}(G)$ with $\langle f \rangle = 1$, for $f_{\infty}$ given by Theorem \ref{ThmEquilibrium},
\[ \begin{aligned}
\|S_t (f - f_{\infty}) \|_{L^1} \leq \frac{C' \ln(1+t)^{n+1}}{(1+t)^{n}} \|f-f_{\infty}\|_{{m_n}}.
\end{aligned} \]
\end{coroll}

We make several remarks on those results.

\begin{rmk}[About question \textit{iii}]
The constants $C$, $C'$ in Theorem \ref{ThmMain} and Corollary \ref{MainCorol} are explicit (constructive) in the easy case of the unit disk. We believe that for any given domain $D$, one may find explicit constants using the geometry of $D$. The measure given by Doeblin-Harris condition is the only part of the proof where one may lose the constructive property of the constants, see Remark \ref{RmkExplicitNu} for more details.
\end{rmk}

\color{black}

\begin{rmk}
In general, we do not have $f_{\infty} \in L^1_{\omega_{n+1}}(G)$. In particular, in the case where $\theta \equiv \Theta$ with $\Theta > 0$ constant, $f_{\infty}$ is explicit and given by (\ref{EqDefEquilibrium}), and $f_{\infty} \in L^1_{m_n}(G) \setminus L^1_{\omega_{n+1}}(G)$. Therefore, one cannot apply Theorem \ref{ThmMain} to study the convergence towards equilibrium. This limiting role of the equilibrium distribution when computing a rate of convergence is well-known in the probabilistic version of the Doeblin-Harris theory used in this paper, see for instance Douc, Fort and Guillin \cite{DOUC2009897} and Hairer \cite{Hairer}.
\end{rmk}

\bb 
\begin{rmk}
To complete the previous remark, note that we do not have $f_{\infty} \in L^1_m(G)$ for a weight $m$ such that $m(x,v) \underset{v \to 0}{\sim} \frac{C}{\|v\|^n}$ for some constant $C > 0$. This is a slight drawback of the method, which prevents us from obtaining the optimal rate $\frac{1}{(t+1)^n}$ from Kuo, Liu and Tsai \cite{Kuo2013, Kuo2014, Kuo2015}. However, the rate obtained here is almost optimal in the sense that it is better than $\frac{1}{(1+t)^{n-\epsilon}}$ for all $\epsilon > 0$. Moreover, even in the unit ball of dimension $n$, our result is slightly different from the one of Kuo et al.: the space of initial data is not the same. Their space of initial data $L^{\infty,\mu}(G)$ is included in $L^1_{m_n}(G)$, while the converse is not true, as examplified by the function of $L^1_{m_n}(G) \setminus L^{\infty,\mu}(G)$ given by:
\[  f(x,v) = \frac{1}{|D|} \frac{|\ln(\|v\|)|^{\frac{1}{2}(1.6\frac{n}{n+1}-1)}}{1 + \|v\|^{n+1}}, \qquad  (x,v) \in G.  \]
\end{rmk}

\bb
\begin{rmk}
\label{RmkEquilibrium}
The existence result in theorem \ref{ThmEquilibrium} (in particular in the case where $\theta$ varies) can also be seen as a consequence of the explicit formula for the equilibrium $f_{\infty}$ obtained by Sone, see the monograph \cite[Chapter 2, Section 2.5, Equation (2.48)]{Sone07} in the form of an infinite series. In this paper, we deduce the result from Theorem \ref{ThmMain} and do not make use of this explicit form of $f_{\infty}$.
%One obtains an infinite series by tracing back along the successive specular reflections by characteristic methods, and uses the boundary conditions to show that the flux at the boundary for each rebound satisfies an integral equation and finally reduces to a constant at equilibrim. The infinite series only has one term in the diffuse case ($\alpha \equiv 1$). 
\end{rmk}
\dd

\begin{rmk}
%The weight $\omega_{n+1}$ satisfies $\omega_{n+1}(x,v) \underset{v \to 0}{\sim} \frac{C}{\|v\|^{n+1} \ln(\|v\|)^{1.6}}$ for some $C > 0$. By using spherical/polar coordinates, one can see that the hypothesis $f \in L^1_{\omega_{n+1}}(G)$ is strong if $f$ has mass around the velocity $v = 0$.
The hypothesis $f \in L^1_{m_n}(G)$ is quite general even if $f$ charges 0, since $m_n(x,v) \underset{v \to 0}{\sim} \frac{C}{\|v\|^n \ln(\|v\|)^{\frac{1.6n}{n+1}}}$ for some $C > 0$. For instance if one considers, as in \cite{Aoki201187}, $\theta \equiv 1$ so that $M$ is independent of $x$ and $f \in L^1(G)$ with $0 \leq f(x,v) \lesssim M(v)$ on $D \times \RR^n$, the assumption is satisfied. 
\end{rmk}

%With the same method of proof, one can for instance show that for $n = 3$ and $f \in L^1_{\omega_2} \setminus L^1_{\omega_3}$, with $\omega_2(x,v) = (e^2 + \frac{d(D)}{\|v\|c_4} - \sigma(x,-v))^{2} \ln(e^2 + \frac{d(D)}{\|v\|c_4} - \sigma(x,-v))^{-\frac{5}{4}}$ for all $(x,v)$ in $\bar{G}$, then
%$$ \|S_t(f - f_{\infty})\|_{L^1} \leq C \frac{\ln(1+t)^{3}}{(1+t)^{2}} \|f - f_{\infty}\|_{L^1_{\omega_2}}. $$

\begin{rmk}
 The boundary condition prevents one from considering higher-order moments, with weight exponents of order larger than $n+1$. Hence Theorem \ref{ThmMain} cannot be improved by considering a higher weight $\omega_{n+2}$ such that $\omega_{n+2}(x,v) \underset{v \to 0}{\sim} \frac{C}{\|v\|^{n+2} \ln(\|v\|)^{1.6}}$ for some $C > 0$. Indeed, the boundary condition becomes a limiting factor for the Lyapunov condition that we will use (see Section \ref{SectionSubGeom} below for more details): to be compatible with our proof, a weight $w$ must satisfy, for all $x \in \pD$,
\[ \int_{\RR^n} M(x,v) |v\cdot n_x| w(x,v) dv < \infty. \]
\end{rmk}

 In \cite{BernouFournierCollisionless}, the authors study a similar model with probabilistic tools, more precisely they use a coupling of two Markov processes to derive a rate similar (up to logarithmic factors) to the one of Corollary \ref{MainCorol}. This method allows one to treat, in the space of measures, various choices of $M$ independent of $x$. Indeed, they only assume that $M$ is radial, with a first order moment and that $M$ is lower bounded by a continuous positive function in a ball around $0$. The domain is again assumed to be $C^2$. 
 
 \color{black}

\paragraph{Comparison with absorbing boundary condition.} We conclude the paper by studying the following close problem
\begin{equation}
\label{Problem2}
\left\{ 
	\begin{aligned}
    \partial_t f + v \cdot \nabla_x f  &= 0, \hspace{5.4cm}  (t,x,v) \in \RR_+^* \times G, \\
 \gamma_- f (t,x,v) &= 0, \hspace{5.4cm} (t,x,v) \in \RR_+ \times \partial_- G, \\ 
 f|_{t=0}(x,v)  &= f_0(x,v), \hspace{4.35cm} (x,v) \in G, 
	\end{aligned}
\right.
\end{equation}
i.e., we take $K \equiv 0$ in (\ref{Problem1}). We set, for $\nu > 0$,
\begin{equation}\begin{aligned}
\label{EqDefWeightR}
r_{\nu}(x,v) = (1 + \sigma(x,v))^{\nu}, \qquad (x,v) \in \bar{G}.
\end{aligned}\end{equation}
We refer to Theorem \ref{ThmAbsorbing} for precise results on (\ref{Problem2}). The rough conclusions of the comparison between the two models are the following.
\begin{enumerate}[1)]
\item For very regular initial data, typically if $f \in L^1(G)$ with $f \mathbf{1}_{\{\|v\| \leq \epsilon\}} = 0$ for some $\epsilon > 0$, the convergence rate is exponential in (\ref{Problem2}) while it is only of order $n+1$ (up to log factors) in (\ref{Problem1}), because of the influence of the boundary condition.
\item With the assumption $f$, $g$ in $L^1_{r_{n+1}}(G)$, the convergence rate of $f_t - g_t$, with obvious notations, is polynomial with, roughly, exponent $n+1$ for both problems.
%or rather $f$, $g$ in $L^1_{r_{n+1-}}(G)$ (meaning that the analysis is valid for any $\epsilon > 0$ if $f,g \in L^1_{r_{n+1-\epsilon}}(G)$)
\item More generally, for $f \in L^1_{r_{\nu}}(G)$, $g \in L^1_{r_{\nu-\delta}}(G)$ with $\nu - \delta > 1$, $\delta > 0$, the exponent of the polynomial rate of convergence is $\nu - \delta$ in (\ref{Problem2}). In particular, if $f \in L^1_{r_{n+1-}}(G)$, the exponent of the polynomial convergence rate towards equilibrium is roughly $n+1$ in (\ref{Problem2}) since the equilibrium $0$ belongs to  $L^1_{r_{n+1-}}(G)$ while it is only $n$  (up to log factors) in  (\ref{Problem1}) since the equilibrium $f_{\infty}$ belongs to $ L^1_{r_{n-}}(G) \setminus L^1_{r_{n+1-}}(G)$.
\end{enumerate}

\paragraph{Proof strategy.} Our proof of Theorem \ref{ThmMain} is purely deterministic. While this proof is also self contained, it is adapted from the method introduced in \cite{CanizoMischler}. Let us elaborate on the strategy. The first step towards the obtention of a Harris' theorem is to prove that, setting $\mathcal{L}$ the operator such that the evolution problem (\ref{Problem1}) rewrites as a Cauchy problem,
\begin{equation}\begin{aligned}
\label{CauchyPb}
\partial_t f &= \mathcal{L} f \hspace{1cm} \text{ in } \bar{D} \times \RR^n, \\
f(0,.) &= f_0(.) \qquad \text{ in } G, 
\end{aligned}\end{equation}
we have the inequality
\begin{equation}\begin{aligned}
\label{EqDoeblinLyapunovClassic}
\mathcal{L}^* \omega_{n+1} \leq - \omega_n + \kappa,
\end{aligned}\end{equation} 
with $\kappa > 0$ constant and $\mathcal{L}^*$ the adjoint operator of $\mathcal{L}$, and that such inequality also holds by considering various couples of weights instead of $(\omega_{n+1}, \omega_n)$. It turns out that, since in our model the whole dissipative component is localized at the boundary, (\ref{EqDoeblinLyapunovClassic}) is very difficult and perhaps impossible to obtain. On the other hand, using that 
\[ v \cdot \nabla_x \sigma(x,v) = -1 \qquad \text{ in } G, \] one can prove an integrated version of (\ref{EqDoeblinLyapunovClassic}), namely that there exist $C_1, b_1 > 0$ such that for all $T > 0$, $f \in L^1_{\omega_{n+1}}(G)$,
\begin{equation}\begin{aligned}
\label{EqIntegratedLyapunovIntro}
\|S_T f\|_{\omega_{n+1}} + C_1 \int_0^T \|S_s f\|_{\omega_{n}} ds \leq \|f\|_{\omega_{n+1}} + b_1(1+T) \|f\|_{L^1},
\end{aligned}\end{equation}
and that this inequality also holds for various couples of weights. 

As a second step, we prove a positivity result (Doeblin-Harris condition) for the semigroup $(S_t)_{t \geq 0}$, Theorem \ref{ThmDoeblinHarris}. By following the characteristics of (\ref{Problem1}) backward, we prove that there exists $R_0 > 0$ such that for all $R > R_0$,  there exist $T(R) > 0$ and a non-negative measure $\nu$ on $\bar{D} \times \RR^n$ with $\nu \not \equiv 0$ such that for all $(x,v) \in G$,
\begin{equation}\begin{aligned} 
\label{EqDoeblinIntro} 
S_{T(R)} f(x,v) \geq \nu(x,v) \int_{\{(y,w) \in D \times \RR^n: \sigma(y,w) \leq R\}} f(y,w) dw dy. 
\end{aligned}\end{equation}
The measure $\nu$ depends on $D$ and whether or not it is constructive is the key point for question iii), see Remark \ref{RmkExplicitNu} below. As already mentioned, if $\nu$ is explicit, the constants $C, C'$ of Theorem \ref{ThmMain} and Corollary \ref{MainCorol} are constructive. 

To obtain the proof of Theorem \ref{ThmMain}, we assume without loss of generality that $g = 0$ and that $f \in L^1_{\omega_{n+1}}(G)$ with $\langle f \rangle = 0$. We fix $T > 0$ large enough and introduce some modified norm \[ \vertiii{.}_{\omega_{n+1}} = \|.\|_{L^1} + \beta \|.\|_{\omega_{n+1}} + \alpha \|.\|_{\omega_n}, \]
for two well-chosen constants $\alpha, \beta > 0$ depending on $T$. We prove, with the help of (\ref{EqIntegratedLyapunovIntro}) and of the Doeblin-Harris condition, that
\begin{equation}\begin{aligned}
\label{EqNormOmegaIntro}
\vertiii{S_T f}_{\omega_{n+1}} \leq \vertiii{f}_{\omega_{n+1}}. 
\end{aligned}\end{equation}
We then introduce some further weights $w_0, w_1$ such that $1 \leq w_0 \leq w_1 \leq \omega_{n+1}$. With a similar argument, we find that, for some modified norm $\vertiii{.}_{w_1}$, for $T$ as above and $\tilde{\alpha} > 0$ well-chosen,
\begin{equation}\begin{aligned}
\label{EqNormw1Intro}
\vertiii{S_T f}_{w_1} + 2 \tilde{\alpha} \|f\|_{w_0} \leq \vertiii{f}_{w_1}.
\end{aligned}\end{equation}
We use repeatedly (\ref{EqNormOmegaIntro}) and (\ref{EqNormw1Intro}), along with the inequalities satisfied by the weights, to conclude.  

Theorem \ref{ThmEquilibrium} is obtained from Theorem \ref{ThmMain} and a refined version of (\ref{EqNormOmegaIntro}) with the couple $(m_{n+1},m_n)$ instead of $(\omega_{n+1},\omega_n)$. Once Theorem \ref{ThmEquilibrium} is established, Corollary \ref{MainCorol} follows from Theorem \ref{ThmMain} by an interpolation argument. 

\paragraph{Proof strategy for the study of (\ref{Problem2}).} 
To compute the convergence rate towards equilibrium of (\ref{Problem2}), we use a method  introduced by Hairer \cite{Hairer} which is much more direct than the previous strategy. This proof can not be easily applied to study (\ref{Problem1}) because of the boundary condition and its impact on the derivation of the Lyapunov inequality. On the other hand the strategy to prove Theorem \ref{ThmMain} can not be adapted easily here because the Doeblin-Harris condition, Theorem \ref{ThmDoeblinHarris}, does not hold. The proof in the case of an exponential weight is a straightforward application of Gronwall's lemma. In the polynomial case, i.e. when the initial data $f \in L^1_{r_{\nu}}(G)$ for some $\nu > 1$, the idea is to prove that,
\[ \mathcal{B}^* r_{\nu} \leq - \phi(r_{\nu}),\]
where $\mathcal{B}$ is the generator of the semigroup which can be associated to (\ref{Problem2}) and where $\phi(x) = \nu x^{\frac{\nu-1}{\nu}}$ for all $x \geq 1$ is a concave function. We then define, for $t \geq 0$, $u \geq 1$, $$\psi(t,u) = (H(u) + t + 1)^{\nu}, $$ with $H(u) = \int_1^u \frac1{\phi(s)} ds = u^{\frac{1}{\nu}} - 1$ for all $u \geq 1$ and prove that $t \to \|S_t f\|_{\psi(t,r_{\nu})}$ is non-increasing in $\RR_+$ using the differential properties of $\psi$, where $(S_t)_{t\geq 0}$ is now the semigroup associated to (\ref{Problem2}). To conclude, we have for all $t \geq 0$,
\[ (t+1)^{\nu} \|S_t f\|_{L^1} \leq \|S_t f\|_{\psi(t,r_{\nu})} \leq \|f\|_{\psi(0,r_{\nu})} = \|f\|_{r_{\nu}}, \]
and the polynomial rate $(t+1)^{-\nu}$ follows. In both cases, the constants are constructive.
\vspace{.5cm}

\noindent \textbf{Plan of the paper.} In Section \ref{SectionSetting} we introduce a few notations and recall some basic properties of (\ref{Problem1}). In Section \ref{SectionSubGeom} we prove the Lyapunov inequality (\ref{EqIntegratedLyapunovIntro}) for several couples of weights. In Section \ref{SectionDoeblin} we prove the Doeblin-Harris condition satisfied by the semigroup $(S_t)_{t \geq 0}$, (\ref{EqDoeblinIntro}). In Section \ref{SectionInterpolation} we recall some interpolation results for $L^1$-weighted space and give very slight extensions in the case of spaces defined through a projection. The proof of Theorem \ref{ThmMain}, Theorem \ref{ThmEquilibrium} and Corollary \ref{MainCorol} is done in Section \ref{SectionProof} using the previous results. Section \ref{SectionAbsorbing} is devoted to the study of the case of an absorbing boundary condition with the strategy detailed above.

\section{Setting and first properties}
\label{SectionSetting}

\subsection{Notations and associated semigroup}
We first set some notations. 
For any set $B$, we write $\bar{B}$ for the closure of $B$.
 For any space $E$, we write $\CD(E) = C^{1}_c(E)$ the space of test functions ($C^1$ with compact support) on $E$. We write $d\zeta(x)$ for the surface measure at $x \in \pD$. We denote by $\mathcal{H}$ the $n-1$ dimensional Hausdorff measure.

For a function $f \in L^{\infty}([0,\infty); L^1( \bar{D} \times \RR^n)),$ admitting a trace $\gamma f$ at the boundary we write $\gamma_{\pm} f$ for its restriction to $(0, \infty) \times \partial_{\pm} G$. This corresponds to the trace obtained in Green's formula, see Mischler \cite{Mischler99}. If $f$ is a solution to (\ref{Problem1}) with initial data $f_0 \in L^1(G)$ the traces are well-defined, see Arkeryd and Cercignani \cite[Section 3]{ArkerydCercignani1993}. 
\begin{lemma}
The boundary operator $K$ defined by (\ref{EqDefK}) is non-negative and satisfies, for all $t \geq 0$, $x \in \pD$, for all $f$ solution to (\ref{Problem1}) with $f_0 \in L^1(G)$, $f$ regular enough so that both integrals are well-defined,
\begin{equation}\begin{aligned}
\label{EqNormK}
\int_{\{v \cdot n_x > 0\}} K \gamma f(t,x,v) (v \cdot n_x) dv  = \int_{\{v \cdot n_x < 0\}} \gamma f(t,x,v) |v \cdot n_x| dv. 
\end{aligned}\end{equation}
\end{lemma}
\begin{proof} The non-negativity of $K$ is straightforward from (\ref{EqDefK}). Since, for all $x \in \pD$,
\[ \int_{\{v \cdot n_x > 0\}} M(x,v) |v \cdot n_x| dv = 1 \]
by (\ref{Eq2tildec}), we have, for all $t \geq 0$, recalling the notation $\eta_x$ from (\ref{EqEtax}),
\[ \begin{aligned}
\int_{\{v \cdot n_x > 0\}} K \gamma f(t,x,v) (v \cdot n_x) dv &=  \int_{\{v \cdot n_x > 0\}} \alpha(x) M(x,v) |v \cdot n_x| dv \\
&\quad \times \int_{\{v' \cdot n_x < 0\}} \gamma f(t,x,v') |v' \cdot n_x| dv' \\
& \quad + \int_{\{v \cdot n_x > 0\}} (1 - \alpha(x)) \gamma f(t,x,\eta_x(v)) |v \cdot n_x| dv,
\end{aligned} \] 
and, using the involutive change of variable $w = \eta_x(v)$ and that $w \cdot n_x = - v \cdot n_x$
\[ \begin{aligned}
\int_{\{v \cdot n_x > 0\}} K \gamma f(t,x,v) (v \cdot n_x) dv &= \alpha(x) \Big(\int_{\{v' \cdot n_x < 0\}} \gamma f(t,x,v') |v' \cdot n_x| dv' \Big) \\
&\qquad + (1 - \alpha(x)) \int_{\{v \cdot n_x < 0\}}  \gamma f(t,x,v) |v \cdot n_x| dv.
\end{aligned} \]
The result follows.
\end{proof}

As a consequence, $\|K\| = 1$ and (\ref{Problem1}) is well posed in the $L^1$ setting, see Arkeryd and Cercignani \cite[Theorem 3.6]{ArkerydCercignani1993}. Therefore we can associate to the equation a strongly continuous semigroup $(S_t)_{t \geq 0}$ of linear operators, such that for all $f_0 \in L^1(G)$, for all $t \geq 0$, $S_t f_0 = f(t,.)$ is the unique solution in $L^\infty([0,\infty); L^1(\bar{D} \times \RR^n))$ to (\ref{Problem1}) taken at time $t$. 
Decay properties of the equation will be studied at the level of this semigroup. 

\subsection{Positivity and mass conservation}

We gather in the next theorem several key properties of (\ref{Problem1}).
For $f \in L^1(G)$, recall the notation $\langle f \rangle$ from (\ref{EqDefNotationMass}).

\begin{thm}
\label{ThmPositivityL1contraction}
Let $f \in L^1(G)$. For all $t \geq 0$, $\langle S_t f \rangle = \langle f \rangle. $ Moreover, we have
\[ \|S_t f\|_{L^1} \leq \|f\|_{L^1}, \]
and if $f$ is non-negative, so is $S_t f$.
\end{thm}

\begin{proof}
We sketch the proof with the additional assumption that $f$ and the trace $\gamma f$ are sufficiently regular so that all the integrals are well-defined, and refer to \cite{ArkerydCercignani1993} for a rigorous derivation of the result. 

\textbf{Step 1.} We write $f(t,x,v)$ for $S_t f(x,v)$ for all $(t,x,v) \in [0,\infty) \times G$, $\gamma f$ for the corresponding trace on $(0, \infty) \times \pD \times \RR^n$. Using Green's formula, we have, for all $t \geq 0$,
\[ \begin{aligned}
\frac{d}{dt} \int_{G} f(t,x,v) dv dx = - \int_{G} v \cdot \nabla_x f(t,x,v) dv dx = \int_{\pD \times \RR^n} \gamma f(t,x,v) (v \cdot n_x) dv d\zeta(x),
\end{aligned} \]
recalling that $n_x$ is pointing towards $D$ for all $x \in \pD$.
Since $\gamma_- f = K \gamma_+ f$, we conclude by (\ref{EqNormK}) that
\[ \frac{d}{dt} \langle S_t f \rangle = 0. \]
%i.e. for all $t \geq 0$, $f \in L^1(G)$,
%$$ \int_{G } S_t f(x,v) dv dx = \int_{G} f(x,v) dv dx. $$

\vspace{.5cm}

\textbf{Step 2.} To establish the contraction result, note first that, by triangle inequality, 
for almost all $t \geq 0$, $x \in \pD$, 
\[ \begin{aligned}
\int_{\{v \cdot n_x > 0\}} &|v \cdot n_x| |K\gamma_+ f|(t,x,v) dv \leq (1-\alpha(x)) \int_{\{v \cdot n_x > 0\}} |v \cdot n_x| |\gamma_+ f|(t,x,\eta_x(v)) dv \\
&\qquad + \alpha(x) \int_{\{v \cdot n_x > 0\}} |v \cdot n_x| M(x,v) \int_{\{v' \cdot n_x < 0\}} |\gamma_+ f|(t,x,v') |v' \cdot n_x| dv'.
\end{aligned} \]
We deduce that
\[ \int_{\{v \cdot n_x > 0\}} |v \cdot n_x| |K\gamma_+ f|(t,x,v) dv \leq \int_{\{v \cdot n_x > 0\}} |v \cdot n_x| K|\gamma_+ f|(t,x,v) dv, \]
and applying (\ref{EqNormK}) with $|\gamma_+ f|$, we conclude that 
\begin{equation}\begin{aligned}
\label{EqFinalStep2Contraction}
 \int_{\{v \cdot n_x > 0\}} |v \cdot n_x| |K\gamma_+ f|(t,x,v) dv \leq \int_{\{v \cdot n_x < 0\}} |v \cdot n_x| |\gamma_+ f|(t,x,v)  dv. 
 \end{aligned}\end{equation}
%We use, in the first term of the right-hand side, the involutive change of variable $w = \eta_x(v)$, and since $|w \cdot n_x| = |v \cdot n_x|$, we obtain
%\[ \begin{aligned}
%\int_{\{v \cdot n_x > 0\}} &|v \cdot n_x| K|\gamma_+ f|(t,x,v) dv =  (1-\alpha(x)) \int_{\{w \cdot n_x < 0\}} |w \cdot n_x| |\gamma_+ f|(t,x,w) dw \\
%&\qquad + \alpha(x) \tilde{c}(x) \int_{\{v \cdot n_x > 0\}} |v \cdot n_x| M(x,v) \int_{v' \cdot n_x < 0} |\gamma_+ f|(t,x,v') |v' \cdot n_x| dv'.
%\end{aligned} \]
%For the second term on the right-hand side, notice that, by definition of $\tilde{c}$ and Fubini's theorem,
%\[ \begin{aligned}
%&\alpha(x) \tilde{c}(x) \int_{\{v \cdot n_x > 0\}} |v \cdot n_x| M(x,v) \int_{\{v' \cdot n_x < 0\}} |\gamma_+ f|(t,x,v') |v' \cdot n_x| dv' \\
%&\qquad = \alpha(x) \int_{\{v' \cdot n_x < 0\}} |\gamma_+ f|(t,x,v') |v' \cdot n_x| dv'. 
%\end{aligned} \]
%We conclude that 
%\begin{equation}\begin{aligned}
%\label{EqKGammaEqual}
% \int_{\{v \cdot n_x > 0\}} |v \cdot n_x| K|\gamma_+ f|(t,x,v) dv = \int_{\{v \cdot n_x < 0\}} |v \cdot n_x| |\gamma_+ f|(t,x,v) dv. 
% \end{aligned}\end{equation}
 
 \vspace{.5cm}
 
\textbf{Step 3.} With similar computations to those of Step 1, one obtains,
\[ \frac{d}{dt} \int_{G} |S_tf|  dv dx = \int_{\partial_+ G} |\gamma_+ f(t)| (v \cdot n_x) dv d\zeta(x) +  \int_{\partial_- G} |\gamma_- f(t)| (v \cdot n_x) dv d\zeta(x). \] According to the boundary condition in (\ref{Problem1}), we have $\gamma_- f(t,x,v) = K \gamma_+ f(t,x,v)$. We conclude that
\[ \frac{d}{dt}\|S_t f\|_{L^1} \leq 0, \]
 using (\ref{EqFinalStep2Contraction}).

\vspace{.5cm}

\textbf{Step 4.} \bb To prove the positivity property, we use the previous results and the fact that $(S_t f)_- = \frac{|S_t f| - S_t f}{2}$. Assuming $f \geq 0$, we have $f_- \equiv 0$ and, for all $t \geq 0$, since $\langle S_t f \rangle = \langle f \rangle$ and since $S_t$ is a contraction in $L^1$,
\[ \begin{aligned}
\|(S_t f)_- \|_{L^1} &= \int_{G} \frac{|S_t f| - S_t f}{2} dv dx \\
&= \frac{1}{2} \Big( \|S_t f\|_{L^1} - \langle S_t f \rangle \Big) 
\\
&\leq \frac12 \Big( \|f\|_{L^1} - \langle f \rangle \Big) = \int_G \frac{|f| - f}{2} dv dx = \|f_-\|_{L^1} = 0 
\end{aligned} \]
and since $(S_t f)_- \geq 0$ almost everywhere (a.e.) on $G$ we deduce that $(S_t f)_- = 0$ a.e. on $G$.
\end{proof}

% Note that the mass conservation is equivalent to the fact that the dual semigroup $(S^*_t)_{t \geq 0}$ satisfies $S^*_t 1 = 1$ for all $t \geq 0$. Indeed, assuming mass conservation,
%\begin{equation}\begin{aligned}
%\int_{G} (S^*_t 1) \phi dv dx  = \int_{G} S_t \phi 1 dv dx = \int_{G} S_t \phi dv dx = \int_{G} \phi dv dx,
%\end{aligned}\end{equation}
%for any test function $\phi \in \mathcal{D}(G)$, so that $S^*_t 1 = 1$, and  the converse result holds.
%

%We recall the positivity property of (\ref{Problem1}) and give a contraction result in the $L^1$ norm.
%
%\begin{thm}[Theorem 3.3 in \cite{ArkerydCercignani1993}]
%\label{ThmL1Contraction}
%If $f$ is nonnegative then for all $t \geq 0$, $S_t f$ is nonnegative.
%Moreover, for all $t \geq 0$, $f \in L^1$,
%$$ \|S_t f\|_{L^1} \leq \|f\|_{L^1}. $$
%\end{thm}

\section{Subgeometric Lyapunov condition}
\label{SectionSubGeom}

In this section, we derive several subgeometric Lyapunov inequalities that will play a key role in our proof of Theorem \ref{ThmMain}. 
Recall the definition (\ref{NotatSigma}) of $\sigma$. We first introduce a notation.
\begin{notation}
We define the map $q$ from $\bar{D} \times \RR^n$ to $\pD$ by
\begin{equation}\begin{aligned}
\label{defq}
q(x,v) &:= x + \sigma(x,v)v,
\end{aligned}\end{equation}
for all $(x,v) \in \bar{D} \times \RR^n$.
\end{notation}
In terms of characteristics of the free transport equation, for $(x,v) \in \bar{D} \times \RR^n$, $q(x,v)$ corresponds to the right limit in $\bar{D}$ of the characteristic with origin $x$ directed by $v$. The real number $\sigma(x,v)$ corresponds to the time at which this characteristic reaches the boundary, if it started from $x$ at time $0$ with velocity $v$ with $x \in D$ or $x \in \pD, v \cdot n_x > 0$. If $x \in \pD$ and $v$ is not pointing towards the gas region (that is, $(x,v)$ is already the right limit of the corresponding characteristic), $q(x,v)$ simply denotes $x$. 

We recall a result on the derivative of $\sigma$ inside $G$ from Esposito, Guo, Kim and Marra \cite[Lemma 2.3]{Esposito2013}. We parametrize locally $D$ by a $C^1$ map $\xi: \RR^n \to \RR$, and $D$ is locally $\{x \in \RR^n, \xi(x) < 0\}$. By definition of $\sigma(x,v)$, for all $(x,v) \in G$, $\xi(x + \sigma(x,v)v) = 0$, and using the implicit function theorem, we find, for all $j \in \{1, \dots, n\}$,
\[ \begin{aligned}
\partial_j \xi + \sum_{i =1}^n \partial_i \xi \frac{\partial \sigma(x,v)}{\partial x_j} v_i = 0.
\end{aligned} \]
Rearranging the terms and by definition of $n_{q(x,v)}$, we have:
\[ \begin{aligned}
\frac{\partial \sigma(x,v)}{\partial x_j} = - \frac{(n_{q(x,v)})_j}{v \cdot n_{q(x,v)}},
\end{aligned} \]
so that  
\begin{equation}\begin{aligned}
\label{EqDerivativeSigma}
\nabla_x \sigma(x,v) = - \frac{n_{q(x,v)}}{v \cdot n_{q(x,v)}}, \quad v \cdot \nabla_x \sigma(x,v) = -1.
\end{aligned}\end{equation}
This minus sign can be understood in the following way: since $\sigma(x,v)$ is the time needed for a particle in position $x \in \bar{D}$ with velocity $v \in \RR^n$ at time $t = 0$ to hit the boundary, moving the particle from $x$ along the direction $v$ reduces this time. 
%We now set $\langle x,v \rangle := (1+ \sigma(x,v))$, for any $(x,v) \in \bar{D} \times \RR^n$. 
Recall from (\ref{EqDefc4}) that $c_4 \in (0,1)$ is such that $(1-c_4)^4 = (1-c_0)$. For all $(x,v) \in \bar{D} \times \RR^n$, we set \[\langle x,v \rangle = \Big(e^2 + \frac{d(D)}{\|v\| c_4} - \sigma(x,-v)\Big),\] so that $e^2 \leq \langle x,v \rangle$ and $\langle x,v \rangle \underset{v \to 0}{\sim} \frac{\kappa}{\|v\|}$ for some $\kappa > 0$. Moreover, for all $(x,v) \in \partial_+ G$, since $\sigma(x,-v) \leq \frac{d(D)}{\|v\|}$ by definition of $d(D)$, $c_4$ is chosen is such a way that we have for all $i \in [1,4]$,
\begin{equation}\begin{aligned}
\label{IneqPolynomSubgeom}
(1-c_0)^{\frac{1}{i}} \Big(e^2 + \frac{d(D)}{\|v\| c_4} \Big) \leq (1-c_4) \Big(e^2 + \frac{d(D)}{\|v\| c_4} \Big)
 &\leq e^2 + \frac{d(D)}{\|v\| c_4} - \sigma(x,-v). 
\end{aligned}\end{equation}
We prove the following:

\begin{lemma}
\label{LemmaLyapunov}
For a couple of weights $(m_1, m_0)$ with any of the choices
\begin{enumerate}[(1)]
\item $(m_1, m_0) = (\langle x, v\rangle^i \ln(\langle x, v \rangle)^{- 1-\epsilon},\langle x, v\rangle^{i-1} \ln(\langle x, v \rangle)^{- 1-\epsilon} ), \quad i \in \llbracket 2, n+1\rrbracket, \epsilon \in (0,3),$ 
\item $(m_1, m_0) = (\langle x,v \rangle^i, \langle x, v \rangle^{i-1})$, $\qquad \quad \hspace{3.95cm} i \in \{\frac32, 2, \frac52, \dots, \frac{2n+1}{2}\}$, 
\item $(m_1, m_0) = (\langle x,v \rangle \ln(\langle x,v \rangle)^{0.1}, \ln(\langle x,v \rangle)^{0.1})$,
%\item $(m_1,m_0) = (\langle x,v \rangle^{n} \ln(\langle x,v \rangle)^{-(1+\epsilon) \frac{n}{n+1}}, \langle x,v \rangle^{n-1} \ln(\langle x,v \rangle)^{-(1+\epsilon) \frac{n}{n+1}}), \quad \epsilon \in (0,1)$
\end{enumerate}
there exist $C > 0$, $b > 0$ explicit, depending on $(m_1,m_0)$, such that for all $T > 0$, all $f \in L^1_{m_1}(G)$,
 \begin{equation}\begin{aligned}
 \label{IneqLyapunovwithContraction}
 \|S_T f\|_{m_1} + C \int_0^T \|S_s f\|_{m_{0}} ds &\leq \|f\|_{m_1} + b(1+T)\|f\|_{L^1}. 
\end{aligned}\end{equation}
\end{lemma}

%\begin{lemma}
%
% Let $m_i(x,v) = \langle x,v \rangle^i \ln(\langle x,v \rangle)^{-\frac{5}{4}}, \nu_i = \langle x,v \rangle^i \ln(\langle x,v \rangle)^{-\frac{15}{8}},  i \geq 0$. Then there exist positive constants $(C_i)_{i = 1}^{n+1}$, $(b_i)_{i = 1}^{n+1}$ such that for all $i \in \{1,\dots, n+1\}$, for all $f \in L^1(m_i)$, for all $T > 0$,
%and there exist positive constants $(\tilde{C}_i)_{i = 1}^{n+1}$, $(\tilde{b}_i)_{i = 1}^{n+1}$ such that for all $i \in \{1,\dots, n+1\}$, for all $f \in L^1(\nu_i)$, for all $T > 0$,
%\begin{equation}\begin{aligned}
%\label{IneqLyapunovXlogXN} 
%\|S_T f\|_{\nu_i} + \tilde{b}_i \int_0^T \|S_s f\|_{\nu_{i-1}} ds &\leq \|f\|_{_i} + \tilde{b}_i \int_0^T \|S_s f\|_{L^1} ds, \\
% \|S_T f\|_{\nu_i} + \tilde{b}_i \int_0^T \|S_s f\|_{\nu_{i-1}} ds &\leq \|f\|_{\nu_i} + \tilde{b}_i(1+T)\|f\|_{L^1}, 
%\end{aligned}\end{equation}
%For $i \in \{ \frac32, \frac52\}$ with the weights $\tilde{m}_i(x,v) = \langle x, v \rangle^i$, $i \geq 0$, playing the role of $m_i$. Moreover, we have, for $w_1(x,v) = \langle x,v \rangle \ln( \langle x,v \rangle)^{\frac{3}{8}}$, 
%$w_0(x,v) = \ln(\langle x, v \rangle)^{\frac{3}{4}}$, for some constants $D_1, B_1 > 0$,
%\[ \begin{aligned}
%\|S_T f\|_{w_1} + D_1 \int_0^T \|S_s f\|_{w_0} ds \leq \|f\|_{w_1} + B_1(1+T)\|f\|_{L^1}.
%\end{aligned} \]
%\end{lemma}

\begin{proof}
\textbf{Step 1.} Note that, for all $(x,v) \in G$, according to (\ref{EqDerivativeSigma}) and to the definition of $\langle x,v \rangle,$  $(v \cdot \nabla_x \langle x,v \rangle) = -1$. We treat case (1) first.
For $i \in \llbracket 2, n+1 \rrbracket$, $\epsilon \in (0,3)$,
\[ \begin{aligned}
 (v \cdot \nabla_x)m_1 &= (v \cdot \nabla_x) \big(\langle x,v \rangle^i \ln(\langle x,v \rangle)^{-(1+\epsilon)} \big) \\
 &= i (v \cdot \nabla_x \langle x,v \rangle) (\langle x,v \rangle)^{i-1} \ln(\langle x,v \rangle)^{-(1+\epsilon)} \\
 & \quad + (v \cdot \nabla_x \langle x,v \rangle) (-(1+\epsilon)) (\langle x,v \rangle)^{i-1} \ln(\langle x,v \rangle)^{-(2+\epsilon)} \\
 &= (\langle x,v \rangle)^{i-1} \ln(\langle x,v \rangle)^{-(1+\epsilon)} \Big( -i + \frac{(1+\epsilon)}{ \ln(\langle x,v \rangle)} \Big).
\end{aligned} \]
Finally, $\ln(\langle x,v \rangle) \geq \ln(e^2) = 2$, hence
\[ (v \cdot \nabla_x)m_1 \leq \Big(-i + \frac{1+\epsilon}{2}\Big) m_{0}, \]
and we set $C_i = i - \frac{1+\epsilon}{2} > 0$. 

%We treat the choice $(m_1,m_0)$ corresponding to (1) first. Note that, using $(v \cdot \nabla_x \langle x,v \rangle) = -1$, for $i \geq 1$,
\vspace{.5cm}

\textbf{Step 2.} Let $f \in L^1_{m_1}(G)$. We differentiate the $L^1_{m_{1}}(G)$ norm of $f$, and use Step 1. We first have, since $n_x$ is the unit normal vector pointing towards the gas region, for $T > 0$, by Green's formula,
\[ \frac{d}{dT} \int_{G} |S_T f|m_1 dv dx = \int_{G} |S_T f| (v \cdot \nabla_x m_1 ) dv dx + \int_{\pD \times \RR^n} (n_x \cdot v) m_1 (\gamma |S_T f|) dv d\zeta(x), \]
where we recall that $d\zeta$ denotes the induced volume form on $\pD$. We have, according to \cite[Corollary 1]{Mischler99}, 
\begin{equation}\begin{aligned}
\label{EqCorollMischler}
 |\gamma S_tf(x,v)| = \gamma |S_t f|(x,v) \quad \text{a.e. in } \big((0, \infty) \times \partial_+ G \big) \cup \big((0,\infty) \times \partial_- G\big),
 \end{aligned}\end{equation}
 hence we will not distinguish between both values in what follows.

Applying Step 1 we find, using the boundary condition and (\ref{EqCorollMischler}),
\begin{equation}\begin{aligned}
\label{IneqDiffMGamma0}
\frac{d}{dT} \int_{G} &|S_T f| m_1 dv dx \\
&\leq - C_i \int_{G} |S_T f| m_{0} dv dx + \int_{\pD \times \RR^n} \gamma |S_T f| m_{1} (v \cdot n_x) dv d\zeta(x)  \\
&\leq - C_i \int_{G} |S_T f| m_{0} dv dx   + \int_{\pD} \alpha(x) \int_{\{v \cdot n_x > 0\}} M(x,v) m_{1}(x,v) (v \cdot n_x) \\
&\quad \times  \Big( \int_{\{v' \cdot n_x < 0  \}} \gamma |S_Tf|(x,v') |v' \cdot n_x| dv' \Big) dv d\zeta(x)  \\
&\quad + \int_{\pD} (1-\alpha(x))\int_{\{v \cdot n_x > 0\}} m_{1}(x,v) (v \cdot n_x) \big(\gamma  |S_Tf|(x,\eta_x(v))\big) dv d\zeta(x)  \\
&\quad - \int_{\pD} \int_{\{v \cdot n_x < 0\}} m_{1}(x,v) |v \cdot n_x| \big(\gamma |S_Tf|(x,v) \big) dv d\zeta(x), 
\end{aligned}\end{equation}
with $\eta_x(v) = v - 2 (v \cdot n_x) n_x$ for all $(x,v) \in \pD \times \RR^n$.
We focus on the third and fourth terms of the last inequality of (\ref{IneqDiffMGamma0}).  We use in the third term, for $x \in \pD$ fixed, the involutive change of variable $w = \eta_x(v)$ in the integral in $v$, so that $v = \eta_x(w)$, $|w \cdot n_x| = |v \cdot n_x|$, $\|w\| = \|v\|$ and $w \cdot n_x < 0$ (since $v \cdot n_x > 0$).
Hence, for all $x \in \pD$,
\[ \int_{\{v \cdot n_x > 0\}} \hspace{-.3cm} m_1(x,v) (v \cdot n_x) \big( \gamma |S_T f|(x,\eta_x(v)) \big) dv = \int_{\{v \cdot n_x < 0\}}\hspace{-.3cm} m_1(x,\eta_x(v)) |v \cdot n_x| \big( \gamma |S_T f|(x,v) \big) dv. \]
For $(x,v) \in \partial_+ G$, $\sigma(x,-\eta_x(v)) = 0$,  therefore the sum of the third and fourth terms of (\ref{IneqDiffMGamma0}) reads
\[ \begin{aligned}
A:= &\int_{\pD} \int_{\{v: v \cdot n_x < 0\}} |v \cdot n_x|  (\gamma |S_T f |(x,v))  \Big \{ (1-\alpha(x)) \Big(e^2 + \frac{d(D)}{\|v\| c_4} \Big)^i \ln \Big(e^2 + \frac{d(D)}{\|v\| c_4} \Big)^{-(1+\epsilon)}  \\
&\quad - \Big(e^2 + \frac{d(D)}{\|v\|c_4} - \sigma(x,-v) \Big)^i \ln  \Big(e^2 + \frac{d(D)}{\|v\|c_4} - \sigma(x,-v) \Big)^{-(1+\epsilon)} \Big\}  dv d\zeta(x).
\end{aligned} \]
Note that, using $1 \leq i \leq 4$, we can control for all $(x,v) \in \partial_+ G$ the quantity $I(x,v)$ defined by
\[ \begin{aligned}
& I(x,v) := (1-\alpha(x)) \Big(e^2 + \frac{d(D)}{\|v\| c_4} \Big)^i \ln \Big(e^2 + \frac{d(D)}{\|v\| c_4} \Big)^{-(1+\epsilon)} \\
&\quad  - \Big(e^2 + \frac{d(D)}{\|v\|c_4} - \sigma(x,-v) \Big)^i \ln  \Big(e^2 + \frac{d(D)}{\|v\|c_4} - \sigma(x,-v) \Big)^{-(1+\epsilon)}.
\end{aligned} \]
Indeed, by definition of $c_4$ and since $\alpha(x) \geq c_0$ for all $x \in \pD$,
\[ \begin{aligned}
I(x,v)
 &\leq (1-c_4)^4 \Big(e^2 + \frac{d(D)}{\|v\| c_4} \Big)^i \ln \Big(e^2 + \frac{d(D)}{\|v\| c_4} \Big)^{-(1+\epsilon)} \\ 
 & \quad - \Big(e^2 + \frac{d(D)}{\|v\|c_4} - \sigma(x,-v) \Big)^i \ln  \Big(e^2 + \frac{d(D)}{\|v\|c_4} - \sigma(x,-v) \Big)^{-(1+\epsilon)}  \\
 &\leq \Big( (1-c_4) \Big(e^2 + \frac{d(D)}{\|v\| c_4} \Big) \Big)^i \ln \Big(e^2 + \frac{d(D)}{\|v\| c_4} \Big)^{-(1+\epsilon)}   \\
 &\quad - \Big(e^2 + \frac{d(D)}{\|v\|c_4} - \sigma(x,-v) \Big)^i \ln  \Big(e^2 + \frac{d(D)}{\|v\|c_4} - \sigma(x,-v) \Big)^{-(1+\epsilon)}. 
\end{aligned} \]
With the obvious bound $e^2 + \frac{d(D)}{\|v\| c_4} \geq e^2 + \frac{d(D)}{\|v\| c_4} - \sigma(x,-v)$ for all $(x,v) \in \partial_+ G$, we deduce easily,
\begin{equation}\begin{aligned}
\label{IneqLogSubgeom}
 \ln \Big(e^2 + \frac{d(D)}{\|v\| c_4} \Big)^{-(1+\epsilon)}  \leq \ln  \Big(e^2 + \frac{d(D)}{\|v\|c_4} - \sigma(x,-v) \Big)^{-(1+\epsilon)}. 
 \end{aligned}\end{equation}
 From  (\ref{IneqPolynomSubgeom}) and (\ref{IneqLogSubgeom}) we conclude that
 \[ I(x,v) \leq 0, \] for all $(x,v) \in \partial_+ G$, and finally that
\[ A = \int_{\partial_+ G} |v \cdot n_x| \big(\gamma| S_Tf|(x,v) \big) I(x,v) dv d\zeta(x) \leq 0. \]
%Set $\bar{m}_i(x,v) = (e^2 + \frac{d(D)}{\|v\|})^i \ln(e^2 + \frac{d(D)}{\|v\|})^{-\frac{5}{4}}$ and note that $x \to (e^2+x)^i \ln(e^2 + x)^{-\frac{5}{4}}$ is non-decreasing, as can be seen with a computation analoguous to the one made for $(v \cdot \nabla_x) m_i$. Hence, $m_i(x, \eta_x(v)) \leq \bar{m}_i(x,v)$ for all $(x,v) \in \bar{D} \times \RR^n$, and, recalling the definition of $\sigma(x,v)$ on $\partial_+ G$, (\ref{defsigma}), we have
%\[ \begin{aligned}
%&\frac{d}{dt} \int_{D \times \RR^n} |S_T f| m_{i} dx dv \\
%& \leq - C_i \int_{D \times \RR^n} |S_T f| m_{i - 1} + \int_{\pD} \alpha(x) \int_{v \cdot n_x > 0} c M(v) m_{i} (v \cdot n_x) \Big( \int_{v' \cdot n_x < 0  } |\gamma f(T,x,v')| |v' \cdot n_x| dv' \Big)  dv d\sigma_x \\
%&\qquad + \int_{\pD} (1- \alpha(x)) \int_{v \cdot n_x < 0} |v \cdot n_(x| \big| \gamma f(T,x,v) \big| \Big( m_i(x,\eta_x(v)) - \bar{m}_i(x,v) \Big) dv d\sigma_x,
%\end{aligned} \]
Applying this result to (\ref{IneqDiffMGamma0}) we obtain
\begin{equation}\begin{aligned}
\label{IneqDiffLyapunov}
\frac{d}{dT} \int_{G} |S_T f| m_{1}  dv dx \leq - C_i \int_{G} &|S_T f| m_{0}  dv dx  
+ \int_{\partial_- G} \alpha(x)  M(x,v) m_{1}(x,v) (v \cdot n_x) \\
&\times  \Big( \int_{\{v' \cdot n_x < 0 \}  } \big(\gamma |S_Tf|(x,v') \big) |v' \cdot n_x| dv' \Big)  dv d\zeta(x). 
\end{aligned}\end{equation}

\vspace{.5cm}

\textbf{Step 3.} We focus on the second term on the right-hand side of (\ref{IneqDiffLyapunov}). We have, for all $T > 0$,
\begin{equation}\begin{aligned}
\label{EqRenormalSolution}
 \partial_t |f| + v \cdot \nabla_x |f| = 0, 
 \end{aligned}\end{equation} a.e. on $(0,T) \times D \times \RR^n$. Recall that $n_{\cdot}: x \to n_x$ is a $W^{1,\infty}$ map on $\bar{D}$. We multiply (\ref{EqRenormalSolution}) by $(v \cdot n_x)$ and integrate it over $(0,T) \times D \times \{v \in \RR^n, \|v\| \leq 1\}$ to obtain, using also Green's formula,
\[ \begin{aligned}
0 &= \int_0^T \int_D \int_{\{\|v\| \leq 1\}} \Big((\partial_t + v \cdot \nabla_x) |S_tf(x,v)| \Big) (v \cdot n_x) dv dx dt \\
&= \Big[ \int_{D \times  \{\|v\| \leq 1\}} |S_t f(x,v)| (v \cdot n_x) dv dx \Big]^T_0
\\
&\quad - \int_0^T \int_D \int_{\{\|v\| \leq 1\}} |S_tf(x,v)| \Big( v \cdot \nabla_x(v \cdot n_x) \Big) dv dx dt \\
& \quad - \int_0^T \int_{\{\|v\| \leq 1\}} \int_{\pD} \big(\gamma|S_t f|(x,v) \big) (v \cdot n_x)^2 d\zeta(x) dv dt,
\end{aligned} \]
where the minus sign in the last term comes from our definition of $n_x$ as a vector pointing towards the gas region.  Using the $L^1$ contraction from Theorem \ref{ThmPositivityL1contraction}, we deduce from the previous computation
\bb
\begin{equation}\begin{aligned}
\label{IneqLyapunovIntegralT}
\int_0^T \int_{\pD} \int_{\{v \cdot n_x > 0, \|v\| \leq 1\}}  \big(\gamma |S_Tf|(x,v) \big)& (v \cdot n_x)^2 dv d\zeta(x) dt \leq 2 \int_{G} |f(x,v)| dv dx  \\
&\quad + T \|n_{\cdot}\|_{W^{1, \infty}}  \int_G |f(x,v)| dv dx. 
\end{aligned}\end{equation}
\dd

\noindent As a consequence of the boundary condition, and since $\alpha \geq c_0$ on $\pD$, we obtain, 
\begin{equation}\begin{aligned}
\label{IneqTrace}
&c_0 \int_0^T \int_{\pD} \Big(\int_{\{v \cdot n_x > 0, \|v\| \leq 1\}} M(x,v) (v \cdot n_x)^2 dv \\
& \qquad\times \int_{\{v' \cdot n_x < 0\}} \big(\gamma |S_T f|(x,v') \big) |v' \cdot n_x| dv' \Big) d\zeta(x) dt  \leq (2 + T \|n_{\cdot}\|_{W^{1, \infty}}) \|f\|_{L^1}. 
\end{aligned}\end{equation}
Note that for a fixed $ x \in \pD$, $x \to \int_{\{v \cdot n_x > 0, \|v\| \leq 1\}}  M(x,v) (v \cdot n_x)^2 dv$ is continuous and positive since $x \to M(x,v)$ and $x \to n_x$ are continuous for all $v \in \RR^n$. Since $\pD$ is compact, writing $\Delta = c_0 \underset{x \in \pD}{\min} \int_{\{v \cdot n_x > 0, \|v\| \leq 1\}}  M(x,v) (v \cdot n_x)^2 dv > 0$, we deduce from (\ref{IneqTrace}) that
\begin{equation}\begin{aligned}
\label{IneqFinalTrace}
\Delta \int_0^T \int_{\partial_+ G} \big( \gamma| S_tf|(x,v) \big) |v \cdot n_x|  dv d\zeta(x) dt \leq \max(2, \|n_{\cdot}\|_{W^{1, \infty}}) (1 + T) \|f\|_{L^1}.
\end{aligned}\end{equation}

\vspace{.3cm}

\textbf{Step 4.} We use the previous steps to conclude the proof of case (1). We integrate (\ref{IneqDiffLyapunov}) over $(0,T)$. Using (\ref{IneqFinalTrace}) and $\alpha \leq 1$ on $\pD$, we obtain:
\[ \begin{aligned}
\|S_T &f\|_{m_{1}} + C_i \int_0^T \|S_s f\|_{m_{0}} ds \\
&\leq \|f\|_{m_{1}} + \int_0^T \int_{\pD} \Big(\int_{\{v \cdot n_x > 0\}} M(x,v) m_{1}(x,v) |v \cdot n_x| \Big) \\
& \qquad \qquad \times \Big( \int_{\{v' \cdot n_x < 0\}} \big(\gamma |S_s f|(x,v')\big) |v' \cdot n_x| dv' \Big) dv d\zeta(x) ds.
\end{aligned} \]
Note that, for $(x,v) \in \pD \times \RR^n$, $\sigma(x,-v) \leq \frac{d(D)}{\|v\|}$, so that
\begin{equation}\begin{aligned}
\label{Defam_1}
\int_{\{v \cdot n_x > 0\}} M(x,v) |v \cdot n_x| m_1(x,v) dv &\leq \int_{\{v \cdot n_x > 0\}} \Big(\underset{x \in \pD}{\max} M(x,v) \Big) \|v\| \Big(e^2 + \frac{d(D)}{\|v\| c_4}\Big)^{i}  \\
&\quad \times \ln \Big(e^2 + \frac{d(D)}{\|v\| c_4} - \frac{d(D)}{\|v\|} \Big)^{-(1+\epsilon)} dv := a_{1},
\end{aligned}\end{equation}
where $a_{1}$ is independent of $x$ and $f$ and finite by choice of $i, \epsilon$. Hence
\begin{equation}\begin{aligned}
\label{IneqGammaNormBeforeTrace}
\hspace{-.3cm} \|S_T f\|_{m_{1}} + C_i \int_0^T  \|S_s f\|_{m_{0}} ds \leq \|f\|_{m_{1}} + a_{1} \int_0^T \int_{\partial_+ G} \hspace{-.5cm} \big(\gamma |S_s f|(x,v') \big) |v' \cdot n_x| dv' d\zeta(x) ds.
\end{aligned}\end{equation}
To conclude, we plug (\ref{IneqFinalTrace}) into (\ref{IneqGammaNormBeforeTrace}) to find
\begin{equation}\begin{aligned}
\|S_T f\|_{m_{1}} + C_i \int_0^T \|S_s f\|_{m_{0}} ds \leq \|f\|_{m_{1}} + \frac{a_1}{\Delta} \max(2, \|n_{\cdot}\|_{W^{1, \infty}}) (1 + T) \|f\|_{L^1}.
\end{aligned}\end{equation}
 Setting $b =  \frac{a_1}{\Delta} \max(2, \|n_{\cdot}\|_{W^{1, \infty}}) $ terminates the proof of case (1). 

\vspace{.5cm}
%\textbf{Step 5.} In case (2), for all $i \in \{2, \dots, n+1\}$, we have for all $(x,v) \in G$, $$(v \cdot \nabla_x) m_1  \leq \Big(-i + \frac{17}{20} \Big) m_0.$$
%We set $C_i = i - \frac{17}{20} > 0$ and replicate the Steps 1 to 4 with the new values $m_1$, $m_0$ and $ C_i$ to conclude. 
% 
% \vspace{.5cm}
\textbf{Step 5.}
In case (2), for all $(x,v) \in G,$ $i \in \{\frac32, 2, \dots, \frac{2n+1}{2}\}$, we have
\[(v \cdot \nabla_x) m_1 = v \cdot \nabla_x (\langle x,v \rangle^i) = -i \langle x,v \rangle^{i-1} = - i m_{0},\]
so that we can replicate the previous Steps 1 to 4 with the choice $C_i = i$ and a new value $a_1$ for (\ref{Defam_1}).

\vspace{.5cm}
\textbf{Step 6.}
For case (3), for $\alpha = 0.1$, for all $(x,v) \in G$, 
\[ \begin{aligned}
v \cdot \nabla_x (m_1(x,v)) &= - \ln(\langle x, v \rangle)^{\alpha} - \alpha \ln(\langle x, v \rangle)^{\alpha - 1} \\
 &= - \ln(\langle x, v \rangle)^{\alpha} (1 + \alpha \ln( \langle x, v \rangle)^{-1}) \\ &\leq - \ln(\langle x, v \rangle)^{\alpha} = -m_0(x,v),
\end{aligned} \]
so that again the previous proof can be replicated with the value $C = 1$ and a new value $a_1$ for (\ref{Defam_1}).
%
%\vspace{.5cm}
%\textbf{Step 7.} For case (4), since $n \geq 2$, we have
%\[ \begin{aligned}
%v \cdot \nabla_x (m_1(x,v)) &= - n \langle x,v \rangle^{n-1} \ln(\langle x,v \rangle)^{-1.6 \frac{n}{n+1}} + \frac{1.6 n }{n+1} \langle x,v\rangle^{n-1} \ln(\langle x,v \rangle)^{-(1.6) \frac{n}{n+1} - 1} \\
%&= - m_0(x,v) \Big(n - 1.6 \frac{n}{(n+1)\ln(\langle x,v \rangle)} \Big) \leq - C m_0(x,v),
%\end{aligned} \]
%with $C = n - 0.8 \frac{ n }{n+1} > 0$, so that similar arguments apply.
\end{proof}

\section{Doeblin-Harris condition}
\label{SectionDoeblin}
Recall that $D$ is a $C^2$ bounded domain. In this section,
we prove the Doeblin-Harris condition, Theorem \ref{ThmDoeblinHarris}. For any two points $x$ and $y$ at the boundary $\pD$ of $D$, we write
\[  ]x,y[ = \{tx + (1-t)y, t \in ]0,1[\}. \]
\begin{defi}
\label{DefiSee}
For $(x,y) \in (\pD)^2$, we write $x \leftrightarrow y$ and say that $x$ and $y$ see each other if $]x,y[ \subset D$, $n_x \cdot (y-x) > 0$, $n_y \cdot (x-y) > 0$.
\end{defi} 

Since $M$ is radial in the second variable, we write $M(x,r) = \frac{\tilde{c}(x)}{(2\pi \theta(x))^{\frac{n}{2}}} e^{-\frac{r^2}{2 \theta(x)}}$ for all $r \in \RR$, $x \in \pD$ see (\ref{EqDefM}) for the definition of $\tilde{c}$, so that $M(x,v) = M(x,\|v\|)$ for all vector $v \in \RR^n$. Possible ambiguity can always be solved by checking the living space of the variable considered.

\noindent We will crucially use this result on $C^1$ bounded domains from Evans:
\begin{prop}[Proposition 1.7 in \cite{evans2001}]
\label{PropEvans}
For all $C^1$ bounded domain $C$, there exist an integer $P$ and a finite set $\Delta' \subset \partial C$ for which the following holds: for all $z', z'' \in \partial C,$ there exist $z_0, \dots, z_P$ with $z' = z_0$, $z'' = z_P$, $\{z_1, \dots z_{P-1}\} \subset \Delta',$ and $z_k \leftrightarrow z_{k+1} $ for $0 \leq k \leq P-1$.
\end{prop}

\noindent We now state the main result of this section. Recall that $(S_t)_{t \geq 0}$ is the semigroup associated to (\ref{Problem1}) as introduced in Section \ref{SectionSetting}.

%\noindent We introduce a few notations to simplify the next computations.
%
%Consider $(v_n)_{n \geq 1}$ a sequence in $\RR^n$, $x \in D$. With the previous notations for $q$ and $\sigma$, we define
%
%$$q_1(x,v_1) = q(x,v), \quad q_{n+1}(x,v_1, \dots, v_{n+1}) = q(q_n(x, v_1, \dots, v_n), v_{n+1}), n \geq 2, $$
%the sequence $q_1(x,v_1), q_2(x,v_1,v_2), \dots... $ is the sequence of successive points of collisions with the boundary of a particle starting at $x$ and having velocity $v_1$ until its first collision with the boundary, picking then the velocity $v_2$ after this collision, until the next collision, and so on...
%We also define:
%$$\sigma_1 (x,v_1) = \sigma(x,v_1), \quad \sigma_{n+1}(x,v_1, \dots, v_{n+1}) = \sigma(q_n(x,v_1, \dots, v_n), v_{n+1}), n \geq 2, $$
%which gives the sequence of the successive times between two collisions at the boundary. 

\begin{thm}[Doeblin-Harris condition]
\label{ThmDoeblinHarris}
For any $R > 0$, there exist $T(R) > 0$ and a non-negative measure $\nu$ on $G$ with $\nu \not \equiv 0$ such that for all $(x,v)$ in $G$, for all $f_0 \in L^1(G), f_0 \geq 0$, 
\begin{equation}\begin{aligned}
\label{DoeblinHarris}
S_{T(R)}f_0(x,v) \geq \nu(x,v) \int_{B_R} f_0(y,w) dw dy,
\end{aligned}\end{equation}
with $B_R = \{(y,w) \in G: \sigma(y,w) \leq R \}$. Moreover there exists $\kappa > 0$ such that for all $R > 0$, $T(R) = \kappa R$.
\end{thm}

\begin{proof}
We only treat the case $n = 3$, as the case of $n=2$ follows from similar (and easier) computations. For all $t \geq 0$, $(x,v) \in \bar{D} \times G$, we write $f(t,x,v) = S_t f_0(x,v)$. To lighten the notations, we write $f(t,x,v) = \gamma f(t,x,v)$ for all $(t,x,v) \in (0, \infty) \times \pD \times \RR^n$. Recall that this trace is well-defined, see Section \ref{SectionSetting}.

\vspace{.5cm}

\textbf{Step 1.} We let $(t,x,v) \in (0,\infty) \times G$ and compute a first inequality for $f(t,x,v)$. Recall the definition of $\sigma$, (\ref{NotatSigma}) and $q$, (\ref{defq}). From the characteristic method we have  
\[ f(t,x,v) = f_0(x - tv, v) \mathbf{1}_{\{t < \sigma(x,-v)\}} + f(t - \sigma(x,-v), q(x,-v),v) \mathbf{1}_{\{t \geq \sigma(x,-v)\}}. \]
Set $y_0 = q(x,-v)$, $\tau_0 = \sigma(x,-v)$. We have, using the boundary condition and the characteristics of the free-transport equation, along with the positivity of $f_0$, with $c_0$ given by (\ref{EqHypoAlpha}),
\[ \begin{aligned}
f(t,x,v) &\geq \mathbf{1}_{\{\tau_0 \leq t\}} f(t - \tau_0, y_0, v) \\
 &\geq \mathbf{1}_{\{\tau_0 \leq t\}} c_0 M(y_0,v) \int_{\{v_0 \in \RR^n, v_0 \cdot n_{y_0} < 0\}} f(t - \tau_0, y_0, v_0) |v_0 \cdot n_{y_0}| dv_0 \\
& \geq \mathbf{1}_{\{\tau_0 \leq t\}} c_0 M(y_0,v) \int_{\{v_0 \cdot n_{y_0} < 0\}} f(t - \tau_0 - \sigma(y_0, -v_0), q(y_0, -v_0), v_0) \\
&\quad \times \mathbf{1}_{\{\tau_0 + \sigma(y_0, -v_0) \leq t\}} |v_0 \cdot n_{y_0}| dv_0 \\
&\geq \mathbf{1}_{\{\tau_0 \leq t\}} c_0^2  M(y_0,v) \int_{\{v_0 \cdot n_{y_0} < 0\}} M(q(y_0,-v_0),v_0) \mathbf{1}_{\{\tau_0 + \sigma(y_0, -v_0) \leq t\}} |v_0 \cdot n_{y_0}| \\
& \quad \times \int_{\{v_1 \cdot n_{q(y_0, -v_0)} < 0\}} f(t - \tau_0 - \sigma(y_0, -v_0), q(y_0, -v_0), v_1) |v_1 \cdot n_{q(y_0, -v_0)}| dv_1 dv_0.
\end{aligned} \]
We write $v_0$ in spherical coordinates $(r, \phi, \theta) \in \RR_+ \times [-\pi, \pi] \times [0,\pi]$ in the space directed by the vector $n_{y_0}$. We write $u = u(\phi, \theta)$ for the unit vector corresponding to the direction of $v_0$. The condition $v_0 \cdot n_{y_0} < 0$ is equivalent to $\phi \in (-\frac{\pi}{2},\frac{\pi}{2})$ and we obtain, using also that $q(y_0,-v_0) = q(y_0,-u)$ as it is independent of $\|v_0\|$,
\[ \begin{aligned}
f(t,x&,v) \geq \mathbf{1}_{\{\tau_0 \leq t\}} c_0^2 M(y_0,v) \int_0^{\infty} \int_{-\pi/2}^{\pi/2} \int_0^{\pi} M(q(y_0,-u),r) \mathbf{1}_{\{\tau_0 + \frac{\sigma(y_0, -u)}{r} \leq t \} } |u \cdot n_{y_0}| \sin(\theta)  \\
& \quad \times r^3    \int_{\{v_1 \cdot n_{q(y_0, -u)} < 0\}} f(t - \tau_0 - \frac{\sigma(y_0, -u)}{r}, q(y_0, -u), v_1) |v_1 \cdot n_{q(y_0, -u)}| dv_1 d\theta d\phi  dr.
\end{aligned} \]
We now use the change of variable $(y_1, \tau_1) = (q(y_0,-u), \sigma(y_0,-ru))$. The inverse of the determinant of the Jacobian matrix was derived in Esposito et al. \cite[Lemma 2.3]{Esposito2013} and is given by (in the case where $y_1 \leftrightarrow y_0$)
\[ \begin{aligned}
\frac{\tau_1^3 r \sin(\theta) |\partial_3 \xi(y_1)|}{|n_{y_1} \cdot u| |\nabla_x \xi(y_1)|},
\end{aligned} \]
where $\xi$ is the $C^1$ function that locally parametrizes $D$, such that $D = \{y: \xi(y) < 0\}$, with the further assumption (which can be done without loss of generality) that $\partial_3 \xi(y_1) \ne 0$. Finally $u$ is the unit vector giving the direction going from $y_1$ to $y_0$, hence
\[ u = \frac{y_0 - y_1}{\|y_0 - y_1\|} \qquad \text{ and } \qquad 
 r = \frac{\|y_1 - y_0\|}{\tau_1}. \]
Setting, for $a \in \pD$, \[U_a := \{y \in \pD, y \leftrightarrow a\},\] we obtain from this change of variables
\[ \begin{aligned}
f(t,x,v) &\geq \mathbf{1}_{\{\tau_0 \leq t\}} c_0^2 M(y_0,v) \int_0^{t - \tau_0} \int_{U_{y_0}} M\Big(y_1, \frac{y_1 - y_0}{\tau_1}\Big) |u \cdot n_{y_0}| \frac{\|y_1 - y_0\|^2}{\tau_1^5}  \\
& \quad \times |u \cdot n_{y_1}|  \frac{|\nabla_x \xi(y_1)|}{|\partial_3 \xi(y_1)|} \int_{\{v_1 \cdot n_{y_1} < 0\}} f(t - \tau_0 - \tau_1, y_1, v_1) |v_1 \cdot n_{y_1}| dv_1 dy_1 d\tau_1 \\
&\geq \mathbf{1}_{\{\tau_0 \leq t\}} c_0^2 M(y_0,v) \int_0^{t - \tau_0} \int_{U_{y_0}} M\Big(y_1, \frac{y_1 - y_0}{\tau_1}\Big) |(y_1 - y_0) \cdot n_{y_0}| |(y_0 - y_1) \cdot n_{y_1}|   \\
& \quad \times \frac{1}{\tau_1^5} \frac{|\nabla_x \xi(y_1)|}{|\partial_3 \xi(y_1)|} \int_{\{v_1 \cdot n_{y_1} < 0\}} f(t - \tau_0 - \tau_1 - \sigma(y_1, -v_1), q(y_1, -v_1), v_1) |v_1 \cdot n_{y_1}|  \\
&\quad \times \mathbf{1}_{\{\sigma(y_1, -v_1) + \tau_1 + \tau_0 \leq t\}} dv_1 dy_1 d\tau_1. 
\end{aligned} \]
 Using again the boundary condition, we have:
\[ \begin{aligned}
f(t,x,v) &\geq \mathbf{1}_{\{\tau_0 \leq t\}} c_0^3 M(y_0,v) \int_0^{t - \tau_0} \int_{U_{y_0}} M\Big(y_1,\frac{y_1 - y_0}{\tau_1}\Big) |(y_1 - y_0) \cdot n_{y_0}| |(y_0 - y_1) \cdot n_{y_1}| \frac{1}{\tau_1^5} \\
& \quad \times \int_{\{v_1 \cdot n_{y_1} < 0\}} |v_1 \cdot n_{y_1}| \mathbf{1}_{\{\sigma(y_1, -v_1) + \tau_1 + \tau_0 \leq t\}} M(q(y_1,-v_1),v_1) \\
&\quad \times \Big( \int_{\{v_2 \cdot n_{q(y_1, -v_1)} < 0\}} f(t - \tau_0 - \tau_1 - \sigma(y_1, -v_1), q(y_1, -v_1), v_2) \\
&\quad \times |v_2 \cdot n_{q(y_1, -v_1)}| dv_2 \Big)  dv_1 d\zeta(y_1) d\tau_1,
\end{aligned} \]
\noindent with $d\zeta$ the surface measure of $\pD$, which is given by $d\zeta(y) = \frac{|\nabla_x \xi(y)|}{|\partial_3 \xi(y)|} dy$ for any $y \in \pD$.

\vspace{.5cm}

\textbf{Step 2.} We use the same method as in Step 1 $P-2$ times and make a change of variable to obtain a first integral over a subset of $D \times \RR^n$. 

Repeating the previous computation $P-2$ times, where $P \in \mathbb{Z}^+$ is given by Proposition \ref{PropEvans}, we obtain,
\[ \begin{aligned}
f(t,x,v) &\geq \mathbf{1}_{\{\tau_0 \leq t\}} c_0^{P+1} M(y_0,v) \int_0^{t - \tau_0} \int_{U_{y_0}} M\Big(y_1,\frac{y_1 - y_0}{\tau_1}\Big) |(y_1 - y_0) \cdot n_{y_0}|  \frac{1}{\tau_1^5} \\
& \quad \times |(y_0 - y_1) \cdot n_{y_1}| \int_0^{t - \tau_0 - \tau_1} \int_{U_{y_1}} M\Big(y_2,\frac{y_2 - y_1}{\tau_2}\Big) |(y_2 - y_1) \cdot n_{y_1}|  \frac{1}{\tau_2^5} \\
& \quad \times |(y_1 - y_2) \cdot n_{y_2}| \times \dots \\
& \quad \times \int_0^{t - \tau_0 - \dots - \tau_{P-1}} \int_{U_{y_{P-1}}} M\Big(y_P,\frac{y_P - y_{P-1}}{\tau_P}\Big) |(y_P - y_{P-1}) \cdot n_{y_{P-1}}|  \\
&\quad \times |(y_{P-1} - y_P) \cdot n_{y_P}| \frac{1}{\tau_P^5} \\
& \quad \times \int_{\{v_P \cdot n_{y_P} < 0\}} f(t - \tau_0 - \dots - \tau_P ,y_P, v_P) |v_P \cdot n_{y_P}|dv_P d\zeta(y_P) d \tau_P \dots d\zeta(y_1) d\tau_1.
\end{aligned} \]
We then use that, on $\{t \geq \tau_0 + \dots + \tau_P\}$,
\[f(t - \tau_0 - \dots - \tau_P, y_P, v_P) \geq f_0(y_P - (t - \tau_0 - \dots - \tau_P)v_P, v_P) \mathbf{1}_{\{t - \tau_0 - \dots - \tau_P - \sigma(y_P, -v_P) \leq 0\}}, \]
and obtain from the previous inequality,
\[ \begin{aligned} 
f(t,x,v) &\geq \mathbf{1}_{\{\tau_0 \leq t\}} c_0^{P+1} M(y_0,v) \int_0^{t - \tau_0} \int_{U_{y_0}} M\Big(y_1,\frac{y_1 - y_0}{\tau_1}\Big) |(y_1 - y_0) \cdot n_{y_0}|  \frac{1}{\tau_1^5} \\
& \quad \times |(y_0 - y_1) \cdot n_{y_1}| \int_0^{t - \tau_0 - \tau_1} \int_{U_{y_1}} M\Big(y_2,\frac{y_2 - y_1}{\tau_2}\Big) |(y_2 - y_1) \cdot n_{y_1}|  \frac{1}{\tau_2^5} \\
& \quad \times |(y_1 - y_2) \cdot n_{y_2}| \times \dots \\
& \quad \times \int_0^{t - \tau_0 - \dots - \tau_{P-1}} \int_{U_{y_{P-1}}} M\Big(y_P,\frac{y_P - y_{P-1}}{\tau_P}\Big) |(y_P - y_{P-1}) \cdot n_{y_{P-1}}|  \\
&\quad \times |(y_{P-1} - y_P) \cdot n_{y_P}| \frac{1}{\tau_P^5} \Big(\int_{\{v_{P} \cdot n_{y_{P}} < 0\}} f_0(y_P - (t - \tau_0 - \dots -\tau_P)v_P, v_P)   \\
& \quad \times |v_P \cdot n_{y_P}| \mathbf{1}_{\{\tau_0 + \dots + \tau_P + \sigma(y_P, -v_P) \geq t\}} dv_P \Big) d\zeta(y_P) d\tau_P \dots d\zeta(y_1) d\tau_1.
\end{aligned} \]
We set $z = \psi(y_P, \tau_P) = y_P - (t - \tau_0 - \dots - \tau_P )v_P$ (i.e. we compute the result of the change of variable from $(y_P, \tau_P)$ to $z$). 
The map $\psi$ is a $C^1$ diffeomorphism with 
\[ \begin{aligned}
\psi: &\{(y_P,\tau_P) \in \pD \times \RR_+: \sigma(y_P, -v_P) > t - \tau_0 - \dots - \tau_P, y_P \leftrightarrow y_{P-1}\}\\
& \quad  \to \{z \in D: q(z,v_P) \leftrightarrow y_{P-1}, \sigma(z,v_P) + \tau_0 + \dots + \tau_{P-1} \leq t\}. 
\end{aligned} \] With this change of variable, $y_P = q(z,v_P)$. Moreover, $t - \tau_0 - \dots - \tau_P = \sigma(z,v_P)$ by definition of $z$, so that
\[ \tau_P = t - \tau_0 - \dots - \tau_{P-1} - \sigma(z, v_P).\] The inverse of the Jacobian is $|v_P \cdot n_{y_P}|$, see Esposito et al. \cite[Lemma 2.3]{Esposito2013}. Therefore,
\[ \begin{aligned}
f(t,x,v) &\geq \mathbf{1}_{\{\tau_0 \leq t\}} c_0^{P+1} M(y_0,v) \int_0^{t - \tau_0} \int_{U_{y_0}} M\Big(y_1,\frac{y_1 - y_0}{\tau_1}\Big) |(y_1 - y_0) \cdot n_{y_0}|  \frac{1}{\tau_1^5} \\
& \quad \times |(y_0 - y_1) \cdot n_{y_1}| \int_0^{t - \tau_0 - \tau_1} \int_{U_{y_1}} M\Big(y_2,\frac{y_2 - y_1}{\tau_2}\Big) |(y_2 - y_1) \cdot n_{y_1}|  \frac{1}{\tau_2^5} \\
& \quad \times |(y_1 - y_2) \cdot n_{y_2}| \times \dots \\
& \quad \times \int_0^{t - \tau_0 - \dots - \tau_{P-2}} \int_{U_{y_{P-2}}} M\Big(y_{P-1},\frac{y_{P-1} - y_{P-2}}{\tau_{P-1}}\Big) |(y_{P-1} -y_{P-2}) \cdot n_{y_{P-2}}|  \\
& \quad \times  |(y_{P-2} - y_{P-1}) \cdot n_{y_{P-1}}|  \frac{1}{\tau_{P-1}^5} \Big(\int_{ D \times \RR^n} \frac{|(y_{P-1} - q(z,v_P)) \cdot n_{q(z,v_P)}|  }{(t - \tau_0 - \dots - \tau_{P-1} - \sigma(z,v_P))^5}    \\
&\quad  \times M\Big(q(z,v_P),\frac{y_{P-1} - q(z,v_P)}{t - \tau_0 - \dots - \tau_{P-1} - \sigma(z,v_P)}\Big)  \\
&\quad \times |(q(z,v_P) - y_{P-1}) \cdot n_{y_{P-1}}| \mathbf{1}_{\{q(z,v_P) \leftrightarrow y_{P-1}\}} \mathbf{1}_{\{\sigma(z,v_P) + \tau_{P-1} + \dots + \tau_0 \leq t\}} \\
&\quad \times f_0(z, v_P) dv_P dz \Big) d\zeta(y_{P-1}) d\tau_{P-1} \dots d\zeta(y_1) d\tau_1.
\end{aligned} \]
Using Tonelli's theorem, we then have
\begin{equation}\begin{aligned}
\label{EqTmpDoeblin}
f(t,x,v) & \geq \mathbf{1}_{\{\tau_0 \leq t\}} c_0^{P+1} M(y_0,v) \int_{ D \times \RR^n} f_0(z, v_P)  \\
& \quad \times \int_0^{t - \tau_0} \int_{U_{y_0}} M\Big(y_1,\frac{y_1 - y_0}{\tau_1}\Big) |(y_1 - y_0) \cdot n_{y_0}|  \frac{1}{\tau_1^5}  \\
& \quad \times |(y_0 - y_1) \cdot n_{y_1}| \int_0^{t - \tau_0 - \tau_1} \int_{U_{y_1}} M\Big(y_2,\frac{y_2 - y_1}{\tau_2}\Big) |(y_2 - y_1) \cdot n_{y_1}|  \frac{1}{\tau_2^5}  \\
& \quad \times |(y_1 - y_2) \cdot n_{y_2}| \times \dots  \\
& \quad \times \int_0^{t - \tau_0 - \dots - \tau_{P-2}} \int_{U_{y_{P-2}}} M\Big(y_{P-1},\frac{y_{P-1} - y_{P-2}}{\tau_{P-1}}\Big) |(y_{P-1} -y_{P-2}) \cdot n_{y_{P-2}}|  \\
& \quad \times  |(y_{P-2} - y_{P-1}) \cdot n_{y_{P-1}}|  \frac{1}{\tau_{P-1}^5} \frac{|(y_{P-1} - q(z,v_P)) \cdot n_{q(z,v_P)}|  }{(t - \tau_0 - \dots - \tau_{P-1} - \sigma(z,v_P))^5}    \\
&\quad  \times M\Big(q(z,v_P),\frac{y_{P-1} - q(z,v_P)}{t - \tau_0 - \dots - \tau_{P-1} - \sigma(z,v_P)}\Big)    \\
&\quad \times |(q(z,v_P) - y_{P-1}) \cdot n_{y_{P-1}}| \mathbf{1}_{\{q(z,v_P) \leftrightarrow y_{P-1}\}} \mathbf{1}_{\{\sigma(z,v_P) + \tau_{P-1} + \dots + \tau_0 \leq t\}}   \\
&\quad \times d\zeta(y_{P-1}) d\tau_{P-1} \dots d\zeta(y_1) d\tau_1  dv_P dz , 
\end{aligned}\end{equation}

\vspace{.5cm}

\textbf{Step 3.}
We choose the value of $t$ and control all the time integrals in (\ref{EqTmpDoeblin}).

 Let $R > 0$ and set
$t = (2P + 2)R$, $\tau_0 \in (R, 2R)$, i.e., for all $(x,v) \in G$ such that $\sigma(x,-v) \not \in (R, 2R)$, we simply set $\nu(x,v) = 0$. Note that for any $R > 0$, one can find a couple $(x,v) \in G$ such that $\sigma(x,-v) \in (R,2R)$. For all $i \in \{1, \dots, P-1\}$, we lower bound the integral with respect to $\tau_i$  by the integral over $(R, 2R)$. We also lower bound the integral with respect to $(z,v_P)$ by an integral over $B_R = \{(z,v_P) \in G: \sigma(z, v_P) \leq R\}$. 
For $\tau_0, \dots, \tau_{P-1} \in (R, 2R)$, $\sigma(z,v_P) \leq R$ and $t = (2P+2)R$, we have first
\[ (2P+2)R - 2P R - R = R \leq t - \tau_0 - \tau_1 - \dots - \tau_{P-1} - \sigma(z,v_P) \leq (2P+2)R - PR = (P+2)R, \]
and thus, with those choices,
\[ \mathbf{1}_{\{\tau_0 + \dots \tau_{P-1} + \sigma(z,v_P) \leq t\}} = 1. \]
Moreover, recalling that for all $i \in \{1,\dots,P-1\}$, the integration interval for $\tau_i$ in the equation (\ref{EqTmpDoeblin}) is $[0, t - \tau_0 - \tau_1 - \dots - \tau_{i-1}]$, and since
\[t - \tau_0 - \tau_1 - \dots - \tau_{i-1} \geq (2P+2)R - 2iR = 2R + 2(P-i)R \geq 2R, \]
the lower bound detailed above using an integral over $[R, 2R]$ for $\tau_i$ is legitimate. We set for all $a > 0$,
\[\underline{M}(a) = \min_{x \in \pD, \tau \in [R, 2R]} M\Big(x,\frac{a}{\tau}\Big) > 0  \quad \text{and} \quad  
 \underline{\underline{M}}(a) = \min_{x \in \pD, \tau \in [R, (P+2)R]} M\Big(x,\frac{a}{\tau}\Big) > 0, \]
 where the positivity is obtained by continuity of $M$ and compactness.
Applying those lower bounds, we obtain from (\ref{EqTmpDoeblin})
\[ \begin{aligned}
f(t,x,v) & \geq \mathbf{1}_{\{\tau_0 \in [R,2R]\}} c_0^{P+1} M(y_0,v) \int_{B_R} f_0(z, v_P)  \\
& \quad \times \int_R^{2R} \int_{U_{y_0}} \underline{M}(\|y_1 - y_0\|) |(y_1 - y_0) \cdot n_{y_0}|  \frac{1}{\tau_1^5} |(y_0 - y_1) \cdot n_{y_1}|  \\
& \quad \times  \int_R^{2R} \int_{U_{y_1}} \underline{M}(\|y_2 - y_1\|) |(y_2 - y_1) \cdot n_{y_1}|  \frac{1}{\tau_2^5} \times |(y_1 - y_2) \cdot n_{y_2}| \times \dots  \\
& \quad \times \int_R^{2R} \int_{U_{y_{P-2}}} \underline{M}(\|y_{P-1} - y_{P-2}\|) |(y_{P-1} -y_{P-2}) \cdot n_{y_{P-2}}| \\
& \quad \times  |(y_{P-2} - y_{P-1}) \cdot n_{y_{P-1}}|  \frac{1}{\tau_{P-1}^5} \frac{|(y_{P-1} - q(z,v_P)) \cdot n_{q(z,v_P)}|  }{((P+2)R)^5}  \\
&\quad  \times \underline{\underline{M}}(\|y_{P-1} - q(z,v_P)\|) |(q(z,v_P) - y_{P-1}) \cdot n_{y_{P-1}}|   \\
&\quad \times  \mathbf{1}_{\{q(z,v_P) \leftrightarrow y_{P-1}\}}  d\zeta(y_{P-1}) d\tau_{P-1} \dots d\zeta(y_1) d\tau_1  dv_P dz.
\end{aligned} \]

Since, $\int_R^{2R} \frac{1}{t^5} dt < \infty$, one finds from (\ref{EqTmpDoeblin}), with $\delta > 0$ explicit, depending on $R$,
\begin{equation}
\label{EqTmpDoeblinStep3}
\begin{aligned}
f(t,&x,v)  \geq \mathbf{1}_{\{\tau_0 \in [R,2R]\}} \delta M(y_0,v) \int_{B_R} f_0(z, v_P)\\
& \quad \times \int_{U_{y_0}} \underline{M}(\|y_1 - y_0\|) |(y_1 -y_0) \cdot n_{y_0}| |(y_0-y_1) \cdot n_{y_1}|  \\
& \quad \times \int_{U_{y_1}} \underline{M}(\|y_2 - y_1\|)|(y_2 - y_1) \cdot n_{y_1}| |(y_1 - y_2) \cdot n_{y_2}| \times \dots   \\
& \quad \times \int_{U_{y_{P-2}}} \underline{M}(\|y_{P-1} - y_{P-2}\|) |(y_{P-1} - y_{P-2}) \cdot n_{y_{P-2}}| |(y_{P-2} - y_{P-1}) \cdot n_{y_{P-1}}|   \\
& \quad \times  |(y_{P-1} - q(z,v_P)) \cdot n_{q(z,v_P)}| |(q(z,v_P) - y_{P-1}) \cdot n_{y_{P-1}}|  \\
&\quad \times  \underline{\underline{M}}(\|y_{P-1} - q(z,v_P)\|) \mathbf{1}_{\{q(z,v_P) \leftrightarrow y_{P-1}\}} d\zeta(y_{P-1}) \dots d\zeta(y_1) dv_P dz. 
\end{aligned}
\end{equation}

\vspace{.5cm}

\textbf{Step 4.} For a couple of points $(a,b) \in (\pD)^2$, we set
\[ \begin{aligned}
h_P(a,b) &= \int_{U_a} \underline{M}(\|y_1 - a\|) |(y_1 - a) \cdot n_{a}| |(a - y_1) \cdot n_{y_1}| \\
& \quad \times \int_{U_{y_1}} \underline{M}(\|y_2 - y_1\|)|(y_2 - y_1) \cdot n_{y_1}| |(y_1 - y_2) \cdot n_{y_2}|  \times \dots  \\
& \quad \times \int_{U_{y_{P-2}}} \underline{M}(\|y_{P-1} - y_{P-2}\|) |(y_{P-1} - y_{P-2}) \cdot n_{y_{P-2}}| |(y_{P-2} - y_{P-1}) \cdot n_{y_{P-1}}|  \\
&\quad \times  |(y_{P-1} -b) \cdot n_{b}| |(b - y_{P-1}) \cdot n_{y_{P-1}}| \underline{\underline{M}}(\|y_{P-1} - b\|) \\
&\quad \times \mathbf{1}_{\{b \leftrightarrow y_{P-1}\}} d\zeta(y_{P-1}) \dots d\zeta(y_1).
\end{aligned} \]
In this step, we want to show that, for all $y_0 \in \pD$, $b \to h_P(y_0, b)$ is lower semicontinuous and positive. 
We can rewrite $h_P$ as
\[ \begin{aligned}
h_P(a,b) = \int_{\{(y_1, \dots, y_{P-1}) \in \bar{D}(a,b)\}} N(a, y_1, \dots, y_{P-1}, b) d\zeta(y_1) \dots d\zeta(y_{P-1}),
\end{aligned} \]
with
\[ \bar{D}(a,b) = \{(y_1, \dots, y_{P-1}) \in (\pD)^{P-1}: y_1 \leftrightarrow a, y_2 \leftrightarrow y_1, \dots ,y_{P-1} \leftrightarrow y_{P-2}, b \leftrightarrow y_{P-1}\}, \]
and
\[ \begin{aligned}
N(a,y_1, \dots y_{P-1},b) &= \underline{M}(\|y_1 - a\|) |(y_1 - a) \cdot n_{a}| |(a - y_1) \cdot n_{y_1}| \\
& \quad \times  \underline{M}(\|y_2 - y_1\|)|(y_2 - y_1) \cdot n_{y_1}| |(y_1 - y_2) \cdot n_{y_2}| \times \dots \\
&\quad \times \underline{M}(\|y_{P-1} - y_{P-2}\|) |(y_{P-1} - y_{P-2}) \cdot n_{y_{P-2}}| |(y_{P-2} - y_{P-1}) \cdot n_{y_{P-1}}| \\
& \quad \times   |(y_{P-1}-b) \cdot n_{b}| |(b - y_{P-1}) \cdot n_{y_{P-1}}| \underline{\underline{M}}(\|y_{P-1} - b\|).
\end{aligned} \] 
By regularity assumption, if $(z_1,z_2) \in (\pD)^2$ with $z_1 \leftrightarrow z_2$, there exists $\epsilon > 0$ such that $B(z_1, \epsilon) \cap \pD \leftrightarrow B(z_2, \epsilon) \cap \pD$, i.e. for all $p \in B(z_1,\epsilon) \cap \pD$, all $q \in B(z_2,\epsilon) \cap \pD$, we have $p \leftrightarrow q$, see \cite[Lemma 38]{BernouFournierCollisionless}. Combining this with the statement of Proposition \ref{PropEvans}, we find that 
\begin{equation}\begin{aligned}
\label{PositiveHausdorffDoeblin}
\mathcal{H}(\bar{D}(a,b)) > 0,
\end{aligned}\end{equation} where we recall that $\mathcal{H}$ denotes the $n-1$ dimensional Hausdorff measure.

We set, for all $a \in \pD$, \[ \bar{D}(a) = \{(y_1, \dots, y_{P-1}) \in (\pD)^{P-1}: y_1 \leftrightarrow a, y_2 \leftrightarrow y_1, \dots ,y_{P-1} \leftrightarrow y_{P-2}\}. \]
For $a \in \pD$ and $(y_1, \dots y_{P-1}) \in \bar{D}(a)$, for all $b \in \pD$ such that $b \leftrightarrow y_{P-1}$, we have $N(a,y_1, \dots, y_{P-1},b) > 0$ according to Definition \ref{DefiSee}. Using (\ref{PositiveHausdorffDoeblin}), one concludes that for all $(a,b) \in (\pD)^2$, $h_P(a,b) > 0$. Moreover, the map $b \to  N(a,y_1, \dots, y_{P-1},b)$ is continuous according to the definition of $\underline{\underline{M}}$ through $M$ and since $z \to n_z$ is continuous.

Note that, according to \cite[Lemma 2.3]{evans2001}, for any $z \in \pD$, the set $U_z$ is open and non-empty. Hence for all $y_{P-1} \in \pD$, $b \to \mathbf{1}_{U_{y_{P-1}}}(b)$ is lower semicontinuous. 
 We conclude that for all $a \in \pD$, $(y_1,\dots, y_{P-1}) \in \bar{D}(a)$,
$b \to N(a,y_1, \dots, y_{P-1},b) \mathbf{1}_{\{y_{P-1} \leftrightarrow b\}}$ is lower semicontinuous.  For $a \in \pD$, $(b_n)_{n \geq 0}$ a sequence of $\pD$ converging towards $b \in \pD$, we obtain
\[ \begin{aligned}
0 < h_P(a,b) &\leq \int_{\bar{D}(a)}  \liminf \limits_{n \to \infty} N(a, y_1, \dots, y_{P-1}, b_n) \mathbf{1}_{\{y_{P-1} \leftrightarrow b_n\}} d\zeta(y_1) \dots d\zeta(y_{P-1}) \\ 
&\leq \liminf \limits_{n \to \infty} h_P(a,b_n), 
\end{aligned} \] 
using Fatou's lemma. Hence $ \pD \ni b \to h_P(a,b)$ is also lower semicontinuous and positive for all $a \in \pD$. 

\vspace{.5cm}

\textbf{Step 5.} We conclude the proof using Step 4. Since $\pD$ is compact, we deduce from the previous step that for all $a \in \pD$,
\[ \mu(a) := \inf_{b \in \pD} h_P(a,b) > 0. \]

With this at hand, we have
\[ \begin{aligned}
f(t,x,v) & \geq \mathbf{1}_{\{\tau_0 \in [R,2R]\}} \delta M(y_0,v) \int_{(z,v_P) \in B_R} f_0(z, v_P) h_P(y_0, q(z, v_P))  dv_P dz \\
&\geq \mathbf{1}_{\{\tau_0 \in [R, 2R]\}} \delta M(y_0,v) \mu(y_0) \int_{B_R} f_0(z,w) dw dz,
\end{aligned} \]
and, recalling that $\tau_0 = \sigma(x,-v)$, $y_0 = q(x,-v)$, we set \[ \nu(x,v) = \delta M(q(x,-v),v) \mu(q(x,-v)) \mathbf{1}_{\{\sigma(x,-v) \in [R,2R]\}},\] and $T(R) = t = (2P + 2)R$ to complete the proof.
\end{proof}

\begin{rmk}
\label{RmkExplicitNu}
Although we use a compactness argument to derive $\mu$, for a given domain $D$, we believe that one may find an explicit lower bound for $h_P$ defined in Step 4 of the previous proof using the geometry of $D$. Note however that this computation might be very difficult. With such constructive lower bound, the constants in Theorem \ref{ThmMain} and Corollary \ref{MainCorol} become explicit. 

As an example of an easy case where an explicit lower bound on $h_P$ can be find, assume that $n = 2$ and $D$ is the unit disk, so that $U_z = \pD \setminus\{z\}$ for all $z \in \pD$. We can clearly take $P = 2$ in Proposition \ref{PropEvans} and we have, for all $(a,b) \in (\pD)^2$,
\[ \begin{aligned}
h_P(a,b) &= \int_{ \pD} \hspace{-.1cm} \underline{M}(\|y-a\|) |(y-a)\cdot n_a| |(a-y) \cdot n_y| \underline{\underline{M}}(\|y-b\|) |(y-b) \cdot n_b| |(b-y) \cdot n_y| dy \\
&\geq \int_{H_{a,b}} \hspace{-.3cm} \underline{M}(\|y-a\|) |(y-a)\cdot n_a| |(a-y) \cdot n_y| \underline{\underline{M}}(\|y-b\|) |(y-b) \cdot n_b| |(b-y) \cdot n_y| dy \\
&\geq \kappa,
\end{aligned} \]  
where \bb 
\[ H_{a,b} = \Big\{y \in \pD, y \cdot a  \wedge y \cdot b \geq \frac{\sqrt{2}}{2} \Big \}, \]
is a set whose Hausdorff measure in $\pD$ is uniformly bounded from below by $\frac1{2}$, and such that for all $y \in H(a,b)$, $d(D) \geq \|y-a\|, \|y-b\| \geq \sqrt{2+\sqrt{2}}$  so that $\kappa$ is a positive constant independent of $a$ and $b$. 
\end{rmk}
\color{black}

Recall that $\langle x, v \rangle = (e^2 + \frac{d(D)}{\|v\| c_4} - \sigma(x,-v))$ for all $(x,v) \in \bar{D} \times \RR^n$, with $c_4 < 1$. We conclude this section by stating a similar result for the level sets of $\langle.,. \rangle$. 
\begin{coroll}
\label{CorolDoeblin}
There exists $R_0 > 0$ such that for any $R \geq R_0$, for $T(R) > 0$ and $\nu$ non-negative measure on $G$ given by Theorem \ref{ThmDoeblinHarris}, for all $(x,v)$ in $G$, for all $f_0 \in L^1(G), f_0 \geq 0$, we have
\begin{equation}\begin{aligned}
\label{DoeblinHarrisWeight}
S_{T(R)} f_0(x,v) \geq \nu(x,v) \int_{\Gamma_R} f_0(y,w) dw dy ,
\end{aligned}\end{equation}
with $\Gamma_R = \{(y,w) \in G, \langle y,w \rangle \leq R \}$. Moreover there exists $\xi > 0$ such that for all $R \geq R_0$, $T(R) = \xi R$.
\end{coroll}

\begin{proof}
Set $R_0 = e^2 + 1$, so that $\lambda(\{(y,w) \in G, \langle y, w \rangle \leq R_0\}) > 0$ where $\lambda$ denotes the Lebesgue measure on $G$. We have, for all $(x,v) \in \bar{D} \times \RR^n$, by definition of $\sigma(x,v)$
\[ \sigma(x,v) + \sigma(x,-v) \leq \frac{d(D)}{\|v\|}, \]
and therefore, using $c_4 < 1$,
\[ \langle x,v \rangle \geq \frac{d(D)}{\|v\|} - \sigma(x,-v) \geq \sigma(x,v). \]
We conclude that for all $R \geq R_0$, $\Gamma_R \subset B_R$ with $\Gamma_R \neq \emptyset$, and the result follows from Theorem \ref{ThmDoeblinHarris}.
\end{proof}

%\begin{coroll}
%\label{CorollaryDoeblinHarris}
%There exists $R_0 > 0$ such that for any $ R \geq R_0$, there exists $T = CR$ for some constant $C > 0$ and a positive measure $\nu \not \equiv 0$ such that for all $(x,w) \in D \times \RR^n$, setting $\mathcal{C} = \{(x,v): CR (1+\sigma(x,v))^k \leq R\}$, $k \geq 1$, we have that $\mathcal{C}$ is non-empty, $\frac{1}{\|v\|} \leq C_2$ for some constant $C_2 > 0$ for all $(x,v) \in \mathcal{C}$ and for all $f_0 \in L^1(D \times \RR^n)$
%$$ S_T f_0(x,w) \geq \nu(x,w) \int_{\mathcal{C}} f_0(y,v) dy dv. $$
%\end{coroll} 
%
%\begin{proof}
%We take $C = (2k+2)$ as in the previous proof, and show that $\mathcal{C} \subset B_R$. Indeed, $\{(x,v): \sigma(x,v) \leq R\} = \{(x,v): (1+\sigma(x,v))^k CR \leq CR(1+R)^k\}$. For $R_0 \geq 1$, for all $R \geq R_0$, $CR(1+R)^k \geq 2CR$ and so $\{(x,v):(1+\sigma(x,v))^k CR \leq 2CR \} = \{(x,v):\sigma(x,v) \leq 2^{\frac{1}{k}} - 1\} \subset  \{(x,v): \sigma(x,v) \leq R \}$.
%\end{proof}

\section{Preliminary interpolation results}
\label{SectionInterpolation}

In this section, we briefly present several results of interpolation theory used in the proof of Theorem \ref{ThmMain}. Those are generalizations of the Riesz-Thorin Theorem for weighted $L^1$ spaces and some of their subspaces. Recall that $G$ denotes $D \times \RR^n$. Recall also that for $(x,v) \in \bar{D} \times \RR^n$, we write $\langle x,v \rangle =  (e^2 + \frac{d(D)}{\|v\| c_4} - \sigma(x,-v))$, with $c_4$ given by (\ref{EqDefc4}). For any weight $w$ on $\bar{D} \times \RR^n$, we set $L^1_{w,0}(G) := \{f \in L^1_w(G), \langle f \rangle = 0\}$ that we endow with the norm $\|.\|_w$ and $L^1_0(G) := \{ f \in L^1(G), \langle f \rangle = 0\}$ which inherits the norm $\|.\|_{L^1}$ from $L^1(G)$. For $A, B$ two Banach spaces with respective norms $\|.\|_A$, $\|.\|_B$ and $T : A \to B$ a linear operator, $ \vertiii{T}_{A \to B}$ denote the operator norm of $T$, i.e. 
\[ \vertiii{T}_{A \to B} = \sup_{v \in A, v \ne 0} \frac{\|Tv\|_B}{\|v\|_A}. \]

We introduce the Maxwellian of temperature 1 given by
\[ M_1(v) = \frac{1}{(2\pi)^{n/2}} e^{-\frac{\|v\|^2}{2}}, \qquad v \in \RR^n. \]

%\begin{defi}
%Given two Banach spaces $A$ and $B$, we say that $C$ is an interpolation space for the couple $(A,B)$ if, for all linear operator $T$ such that there exists $N_1, N_2 > 0$ satisfying
%$$ \vertiii{T}_{A \to A} \leq N_1, \qquad \vertiii{T}_{B \to B} \leq N_2, $$
%we have
%$$ \vertiii{T}_{C \to C} \leq N, $$
%for some constant $C > 0$. 
%\end{defi} 

\begin{lemma}[Interpolation of $L^1$-weighted spaces] \label{LemmaInterpolation}
Let $\phi_1, \phi_2, \tilde{\phi}_1, \tilde{\phi}_2$ four measurable functions on $G$ such that $\phi_1, \phi_2, \tilde{\phi}_1, \tilde{\phi}_2 > 0$ almost everywhere.  Set $A_1 = L^1_{\phi_1}(G),$ $A_2 = L^1_{\phi_2}(G)$, $\tilde{A}_1 = L^1_{\tilde{\phi}_1}(G)$, $\tilde{A}_2 = L^1_{\tilde{\phi}_2}(G)$.  Then, if $T$ is a linear operator from $A_1$ to $\tilde{A}_1$ and from $A_2$ to $\tilde{A}_2$ such that
\begin{equation}\begin{aligned}
\label{defInterpolSpace1}
 \vertiii{T}_{A_1 \to \tilde{A}_1} \leq N_1, \quad \vertiii{T}_{A_2 \to \tilde{A}_2} \leq N_2, 
 \end{aligned}\end{equation}
for some $N_1, N_2 > 0$, for any $\theta \in (0,1)$, for $\phi_{\theta}$, $\tilde{\phi}_{\theta}$ defined on $G$ by $\phi_{\theta} =  \phi_1^{\theta} \phi_2^{1-\theta}$ and for $\tilde{\phi}_{\theta} = \tilde{\phi}_1^{\theta} \tilde{\phi}_2^{1-\theta}$, $T$ is a linear operator from $A_{\theta} := L^1_{\phi_{\theta}}(G)$ to $\tilde{A}_{\theta} = L^1_{\tilde{\phi}_{\theta}}(G)$ satisfying
\begin{equation}\begin{aligned}
\label{defInterpolSpace2} 
\vertiii{T}_{A_{\theta} \to \tilde{A}_{\theta}} \leq N_{\theta},
\end{aligned}\end{equation}
for  $N_{\theta} = N_1^{\theta} N_2^{1 - \theta} > 0. $
\end{lemma}

\begin{proof}
This is obtained by Peetre's K-method of interpolation \cite{peetre1968theory} and is a particular case of the Stein-Weiss Theorem with $p = 1$, see \cite[Theorem 5.4.1]{BerghInterpolationSpaces1976}.
\end{proof}

\begin{coroll}
\label{firstCorollInterpol}
Let $\theta \in (0,1)$ and $A_1, A_2, \tilde{A}_1, \tilde{A}_2, A_{\theta}, \tilde{A}_{\theta}$ defined as in Lemma \ref{LemmaInterpolation}. Assume that there exists a bounded projection $P: (A_i,\tilde{A}_i) \to (A_i',\tilde{A}_i')$ for $i \in \{1,2\}$ with $A_i' \subset A_i$, $\tilde{A}_i' \subset \tilde{A}_i$. Let $A'_{\theta} = (A'_1 + A'_2) \cap A_{\theta}$, $\tilde{A}'_{\theta} = (\tilde{A}'_1 + \tilde{A}'_2) \cap \tilde{A}_{\theta}$. Assume that $T$ is a linear operator from $A'_1$ to $\tilde{A}'_1$ and from $A'_2$ to $\tilde{A}'_2$ with
\[ \vertiii{T}_{A'_1 \to \tilde{A}'_1} \leq N_1, \qquad \vertiii{T}_{A_2' \to \tilde{A}_2'} \leq N_2, \]
for $N_1, N_2 > 0$. Then $T$ is a linear operator from $A'_{\theta}$ to $\tilde{A}'_{\theta}$ and there exists $C > 0$ depending only on $P$ such that
\[ \vertiii{T}_{A'_{\theta} \to \tilde{A}'_{\theta}} \leq C N_1^{\theta} N_2^{1- \theta}. \]

% the previous result holds with $A_1'$, $A_2'$, $A_{\theta}'$ in place of $A_1$, $A_2$, $A_{\theta}$, with possibly $N_{\theta} \ne N_1^{\theta} N_2^{1-\theta}$ in (\ref{defInterpolSpace2}), $N_{\theta} > 0$ constant.
\end{coroll}

\begin{proof}
The couple $(A_1', A_2')$ is a complemented subcouple of $(A_1,A_2)$, and as such a so-called $K$-subcouple for the $K$-method of interpolation. The same thing holds with $(\tilde{A}_1', \tilde{A}_2')$ which is a complemented subcouple of $(\tilde{A}_1, \tilde{A}_2)$. This immediatly gives the result, see \cite[Section 7, Theorem 2.1 and Example 7.1]{janson1993}. We refer to \cite{janson1993} for details about those notions.
\end{proof}

We now turn to a second type of interpolation results in $L^1$ weighted spaces, no more focused on polynomial interpolation.  

\begin{lemma}
\label{LemmaGoulaouic}
For $(y,v) \in G$, let $\phi_1$ defined by $\phi_1(y,v) = (\langle y, v \rangle)$. Let $T$ be a linear operator from $L^1_{\phi_1}(G)$ to $L^1_{\phi_1}(G)$ and  from $L^1(G)$ to $L^1(G)$ such that
\[ \vertiii{T}_{L^1_{\phi_1}(G) \to L^1_{\phi_1}(G)} \leq N_1, \quad \vertiii{T}_{L^1(G) \to L^1(G)} \leq N_2, \]
for some $N_1, N_2 > 0$. Then, for $R(y,v) = \ln(\phi_1(y,v))$, $T$ is a linear operator from $L^1_{R}(G)$ to itself and there exists an explicit $C > 0$ such that
\[ \vertiii{T}_{L^1_R(G) \to L^1_R(G)} \leq C. \]
\end{lemma}

\begin{proof}
From \cite[Chapter 2, Theorem 1]{Goulaouic}, given a weight $\phi_1$, the space $L^1_{\phi_2}(G)$ is an interpolation space (and therefore $ \vertiii{T}_{L^1_{\phi_2}(G) \to L^1_{\phi_2}(G)} \leq C$ for some constant $C > 0$) for the couple $(L^1(G), L^1_{\phi_1}(G))$ if $\phi_2$ satisfies for all $(y,v) \in G$, 
\[ \phi_2(y,v) = \int_0^{\infty} \min(\phi_1(y,v),t) \gamma(dt), \]
for some positive measure $\gamma$ on $(0,\infty)$, $\gamma \not \equiv 0$ and satisfying
\[ \int_0^{\infty} \min(1,t) \gamma(dt) < \infty. \]
The constant $C$ then depends only on $N_1, N_2$ and $\gamma$.
In particular, we consider a measure of the form $\gamma = f \lambda$, with $\lambda$ the Lebesgue measure on $(0,\infty)$, and 
\[ \begin{aligned}
f(t) = \left \{ \begin{array}{ll}
0 &\text{ if } t \in (0,e), \\
\frac{1}{t^2} &\text{ if } t \in (e, \infty). 
\end{array}
\right.
\end{aligned} \]
We then have, for all $(y,v) \in G$, \[\phi_2(y,v) = \int_e^{\phi_1(y,v)} \frac{dt}{t} + \phi_1(y,v) \int_{\phi_1(y,v)}^{\infty} \frac{dt}{t^2}  = \ln(\phi_1(y,v)) = R(y,v), \]
and since $\phi_1(y,v) = \langle y, v \rangle$ for all $(y,v) \in G$, the result follows.
\end{proof}

\begin{coroll}
\label{Coroll2Goulaouic}
Lemma \ref{LemmaGoulaouic} holds when replacing the space $L^1_w(G)$ by $L^1_{w,0}(G)$ for any weight $w$ on $G$ considered, including replacing $L^1(G)$ by $L^1_0(G)$.
\end{coroll}

\begin{proof}
We set, for $f \in L^1(G)$, for all $(x,v) \in G$, $Pf(x,v) = f(x,v) - \frac{M_1(v)}{|D|} \int_{G} f(y,w) dw dy$. Then, for $\phi_1 \geq 1$ defined as in Lemma \ref{LemmaGoulaouic}, we have
\[ \|Pf\|_{L^1} \leq 2 \|f\|_{L^1} \quad  \text{and } \quad \|Pf\|_{\phi_1} \leq (1+c) \|f\|_{\phi_1}, \]
with $c = \int_{G} M_1(v) \frac{\phi_1(x,v)}{|D|}  dv dx< \infty$. The map $P$ is obviously linear, and $P^2 f = Pf$. We conclude as in the proof of Corollary \ref{firstCorollInterpol}  that, setting for all $(x,v) \in G$, $\phi_2(x,v) = \ln( \langle x,v \rangle)$,
\[ L^1_{\phi_2}(G) \cap (L^1_0(G) + L^1_{\phi_1,0}(G)) = L^1_{\phi_2,0}(G), \]
is the interpolation space  required, i.e. is such that for any $T$ linear from $L^1_0(G)$ to itself and from $L^1_{\phi_1,0}(G)$ to itself with 
\[ \vertiii{T}_{L^1_0(G) \to L^1_0(G)} \leq N_1, \qquad \vertiii{T}_{L^1_{\phi_1,0}(G) \to L^1_{\phi_1, 0}(G)} \leq N_2 , \]
for two constants $N_1, N_2 > 0$, $T$ is a linear operator from $L^1_{\phi_2,0}(G)$ to itself and there exists $N > 0$ explicit such that
\[ \vertiii{T}_{L^1_{\phi_2,0}(G) \to L^1_{\phi_2,0}(G)} \leq N. \] 

\vspace{-.6cm}

\end{proof} 

\section{Proof of Theorem \ref{ThmMain}, Theorem \ref{ThmEquilibrium} and Corollary \ref{MainCorol}}
\label{SectionProof}
This section is devoted to the proof of Theorem \ref{ThmMain}, Theorem \ref{ThmEquilibrium} and Corollary \ref{MainCorol}. We recall the notation $\langle x, v \rangle = (e^2 + \frac{d(D)}{\|v\| c_4} - \sigma(x,-v))$ for all $(x,v) \in \bar{D} \times \RR^n$. In this section, the constants are explicit up to the fact that they depend on $\nu$ given by Corollary \ref{CorolDoeblin}. As already stated, $\nu$ itself may not be explicit,  see Remark \ref{RmkExplicitNu}.
%In this section, the constants $C$ and $C' > 0$ are independent of the function $f$ considered, and are allowed to vary from line to line.

In the first subsection, we establish some contraction property for a well-chosen norm. In the second part, we use this property and the previous results to conclude the proof of Theorem \ref{ThmMain}. Subsection \ref{SubsectionProofEquilibrium} is devoted to the proof of Theorem \ref{ThmEquilibrium} and Corollary \ref{MainCorol}. 

\subsection{Contraction property in well-chosen norm}
\label{SubsectionPreliminaryProof}

This subsection is devoted to the proof of the following lemma.

%We recall that the weights $\omega_i$, $ i \in [\![1,4]\!]$ on $\bar{D} \times \RR^n$ are defined by:
%\[ \begin{aligned}
%\omega_i(x,v) &:= \langle x,v \rangle^i \ln(\langle x,v \rangle)^{-1.6}.
%\end{aligned} \] 
%and set, in the case $n = 3$, $$m_3 = \omega_4, \quad  m_2 = \omega_3, \quad m_1 = \omega_2,$$
%and in the case $n = 2$, $$m_3 = \omega_3, \quad  m_2 = \omega_2, \quad m_1 = \omega_1. $$

\begin{lemma}
\label{LemmaContractionm3}
For all $\epsilon \in (0,3)$, setting $\bar{\omega}_{k}(x,v) = \langle x,v \rangle^{k} \ln(\langle x,v \rangle)^{-(1+\epsilon)}$ on $G$ with the value $k \in \llbracket n-1, n+1\rrbracket$  there exists $T_0 > 0$ such that for all $T \geq T_0$, there exist $\beta(T) > 0$, $\alpha = C_3 \beta(T) T$ with $C_3 > 0$ constant such that, for all $f \in L^1_{\bar{\omega}_{n+1}}(G)$ with $\langle f \rangle = 0$,
we have
\begin{equation}\begin{aligned}
\label{EqContractionLemmaIntermediate}
\|S_T f\|_{L^1} + \beta \|S_T f\|_{\bar{\omega}_{n+1}} + \alpha \|S_T f\|_{\bar{\omega}_n} \leq \|f\|_{L^1} + \beta \|f\|_{\bar{\omega}_{n+1}} + \frac{\alpha}{3} \|f\|_{\bar{\omega}_n},
\end{aligned}\end{equation}
so that, setting \[\vertiii{.}_{\bar{\omega}_{n+1}} := \|.\|_{L^1} + \beta \|.\|_{\bar{\omega}_{n+1}} + \alpha \|.\|_{\bar{\omega}_n}, \]
there holds
$ \vertiii{S_T f}_{\bar{\omega}_{n+1}} \leq  \vertiii{f}_{\bar{\omega}_{n+1}}$. Moreover, there exists $M_{n+1} > 1$ such that for all $f \in L^1_{\bar{\omega}_{n+1}}(G)$ with $\langle f \rangle = 0$, \[ \|S_T f\|_{\bar{\omega}_{n+1}} \leq M_{n+1} \|f\|_{\bar{\omega}_{n+1}}. \] 
Finally, setting $\tilde{w}_i(x,v) = \langle x,v \rangle ^{i-\frac12}$ on $G$ with $i \in \{1, \dots, 4\}$, there exists $\tilde{T}_0 > 0$ such that for all $T \geq \tilde{T}_0$, there exists $\tilde{M}_{n+1} > 0$ such that for all $f \in L^1_{\tilde{w}_{n+1}}(G)$ with $\langle f \rangle = 0$,
\[ \|S_T f\|_{\tilde{\omega}_{n+1}} \leq \tilde{M}_{n+1} \|f\|_{\tilde{\omega}_{n+1}}. \]
\end{lemma}

\begin{proof} 

We prove the result on $\bar{\omega}_{n+1}$ first, and explain how to adapt the argument for the second statement at the end of the proof.

\vspace{.3cm} 

\textbf{Step 1.} We use the Lyapunov condition, Lemma \ref{LemmaLyapunov}, case (1), with both $\bar{\omega}_{n+1}$ and $\bar{\omega}_n$ to deduce a new integral inequality. For any $T > 0$, using Lemma \ref{LemmaLyapunov}, with $C_3, C_2, \tilde{b}_3, b_2 > 0$ constant, for all $f \in L^1_{\bar{\omega}_{n+1}}(G)$,
\begin{linenomath}
\begin{subequations}\begin{alignat}{2}
\label{EqLyapunovOmegaPlus1}
\hspace{.4cm }\|S_T f \|_{\bar{\omega}_{n+1}} + C_3 \int_0^T \|S_t f\|_{\bar{\omega}_n} dt &\leq \|f\|_{\bar{\omega}_{n+1}} + \tilde{b}_3 (1+T) \| f\|_{L^1} , \\
\label{EqLyapunovm2}
\text{ and } \|S_T f \|_{\bar{\omega}_n} + C_2 \int_0^T \|S_t f\|_{\bar{\omega}_{n-1}} dt &\leq \|f\|_{\bar{\omega}_n} + b_2 (1+T) \| f\|_{L^1}.
\end{alignat}\end{subequations}
\end{linenomath}
Let $t \in (0,T)$. From (\ref{EqLyapunovm2}) we deduce
\[ \|S_{T-t}S_t f\|_{\bar{\omega}_n} \leq \|S_t f\|_{\bar{\omega}_n} + b_2 (1+T-t) \|S_t f\|_{L^1}, \]
which we rewrite as
\[ \|S_T f\|_{\bar{\omega}_n} -  b_2 (1+T-t) \|S_t f\|_{L^1} \leq \|S_t f\|_{\bar{\omega}_n}. \]
We plug this inside (\ref{EqLyapunovOmegaPlus1}) to obtain
\[ \begin{aligned}
\|S_T f\|_{\bar{\omega}_{n+1}} + C_3 \int_0^T \Big( \|S_T f\|_{\bar{\omega}_n} - b_2 (1 + T - t) \|S_t f\|_{L^1} \Big) dt \leq \|f\|_{\bar{\omega}_{n+1}} + \tilde{b}_3 (1+T) \|f\|_{L^1}.
\end{aligned} \]
Using the $L^1$ contraction result from Theorem \ref{ThmPositivityL1contraction}, we conclude
\begin{equation}\begin{aligned}
\label{IneqFinalM3}
\|S_T f\|_{\bar{\omega}_{n+1}} + C_3 T \|S_T f\|_{\bar{\omega}_n} \leq \|f\|_{\bar{\omega}_{n+1}} + b_3(1+T+T^2)\|f\|_{L^1},
\end{aligned}\end{equation}
with $b_3 > 0$ constant. 
%In the case $n = 2$, the same proof leads to
%\begin{equation}\begin{aligned}
%\label{IneqFinalM3dim2}
%\|S_T f\|_{\bar{\omega}_{n+1}} + \tilde{C}_3 T \|S_T f\|_{\bar{\omega}_n} \leq \|f\|_{\bar{\omega}_{n+1}} + \tilde{b}_3(1+T+T^2)\|f\|_{L^1},
%\end{aligned}\end{equation}
%with two constant $\tilde{C}_3, \tilde{b}_3 > 0$. 

\vspace{.5cm}

\textbf{Step 2.}  From the Doeblin-Harris condition, Theorem \ref{ThmDoeblinHarris}, and more precisely Corollary \ref{CorolDoeblin}, for all $\rho > R_0$, there exist $T(\rho) = \xi \rho$ for some constant $\xi > 0$ and a measure $\nu$ on $G$ with $\nu \not \equiv 0$ such that
\[ S_{T(\rho)}h \geq \nu \int_{\{(x,v) \in G: \langle x, v \rangle \leq \rho\}} h dvdx, \]
for all $h \in L^1(G)$ with $h \geq 0$. 

Recall that by assumption $f$ is such that $f \in L^1_{\bar{\omega}_{n+1}}(G)$, and
$\langle f \rangle = 0.$

\noindent Set for any $\rho \geq R_0$, $\bar{\bar{\omega}}_n(\rho) := \rho^{n} \ln(\rho)^{-(1 + \epsilon)}$ and $\kappa(\rho) = \frac{b_3(1+T+T^2)}{T}(\rho)$. Since $T(\rho) = \xi \rho$ for some constant $\xi > 0$, $\kappa(\rho) \underset{\rho \to +\infty}{\sim} C\rho$ for some $C > 0$. Since $ n \in \{2,3\}$ one can find $\rho_0$ such that  for all $\rho \geq \rho_0$, $\bar{\bar{\omega}}_n(\rho) \geq \frac{12 \kappa(\rho)}{C_3}$. We fix $\rho > \rho_0$, $T = T(\rho) > T(\rho_0)$ for the remaining part of the proof. Note that since $T(\rho) = \xi \rho$ for some given constant $\xi$, any choice of $T > T(\rho_0)$ is possible. We set $A :=  \frac{\bar{\bar{\omega}}_n(\rho)}{4}$, and define, for all $\beta > 0$, the $\beta$-norm by:
\[ \begin{aligned}
\|f\|_{\beta} := \|f\|_{L^1} + \beta \|f\|_{\bar{\omega}_{n+1}}.
\end{aligned} \]
%In the next two steps, we only treat the case $n = 3$, as the case $n=2$ follows from the same demonstration. 
We distinguish two cases. Indeed, we have the alternative
\begin{linenomath}
\begin{subequations}\begin{alignat}{2}
\label{EqAlternative1}
\hspace{.55cm}  \|f\|_{\bar{\omega}_{n}} \leq A \|f\|_{L^1}, \\
\label{EqAlternative2}
\text{ or } \|f\|_{\bar{\omega}_{n}} > A \|f\|_{L^1}.
\end{alignat}\end{subequations}
\end{linenomath}

\vspace{.5cm}

\textbf{Step 3.} We prove a convergence result in the $\beta$-norm in the case of the first alternative, (\ref{EqAlternative1}). Recall that for all $R > 0$, $\Gamma_R = \{(x,v) \in G, \langle x,v \rangle \leq R\}$. Using $\langle f \rangle = 0$, we have for all $(x,v) \in G$, 
\[ \begin{aligned}
S_Tf_{\pm}(x,v) &\geq \nu(x,v) \int_{G} f_{\pm}(x',v') dv' dx' - \nu \int_{\Gamma_{\rho}^c} f_{\pm}(x',v') dv' dx' \\
& \geq \frac{\nu(x,v)}{2} \int_G |f(x',v')| dv' dx'  - \nu(x,v) \int_{\Gamma_{\rho}^c} |f(x',v')| dv' dx' \\
& \geq \frac{\nu(x,v)}{2} \int_G |f(x',v')| dv' dx'  - \frac{\nu(x,v)}{\bar{\bar{\omega}}_n(\rho)} \int_{G}  |f(x',v')| \bar{\omega}_{n}(x',v') dv' dx' \\
& \geq \frac{\nu(x,v)}{2} \int_G |f(x',v')| dv' dx'  - \frac{\nu(x,v)}{4}  \int_{G}  |f(x',v')| dv' dx' \\
&= \frac{\nu(x,v)}{4}  \int_{G}  |f(x',v')| dv' dx' := \eta(x,v),
\end{aligned} \] 
where the third inequality is given by definition of $\Gamma_{\rho}$ and $\bar{\omega}_n \geq 1$, since $\bar{\omega}_{n}(x,v) \leq \bar{\bar{\omega}}_n(\rho)$ for all $(x,v) \in \Gamma_{\rho}$, recalling also that $\langle x,v \rangle \geq e^2$. The last inequality is obtained by condition (\ref{EqAlternative1}). The final equality stands for a definition of $\eta(x,v)$ for all $(x,v) \in G$. Note that $\eta \geq 0$ on $G$.
We deduce,
\[ \begin{aligned}
|S_T f| &= |S_T f_+ - \eta - (S_T f_- - \eta)| \\
		&\leq |S_T f_+ - \eta| + |S_T f_- - \eta| \\
		&= S_T f_+ + S_T f_- - 2 \eta = S_T|f| - 2 \eta,
\end{aligned} \]
and, integrating over $G$, we obtain, using also the mass conservation, that $\eta = \frac{\nu}{4} \|f\|_{L^1}$, and that $\nu$ is non-negative,
\begin{equation}\begin{aligned}
\label{IneqContractionDoeblin}
\|S_T f \|_{L^1} \leq \|f\|_{L^1} - 2 \|\eta\|_{L^1} = ( 1- \frac{\langle \nu \rangle}{2} ) \|f\|_{L^1} = \tilde{\eta} \|f\|_{L^1},
\end{aligned}\end{equation}
with $\tilde{\eta} \in (0, 1)$. Hence, $S_T$ is a strict contraction in $L^1$ in the case where $f$ satisfies (\ref{EqAlternative1}).
We use this result along with (\ref{IneqFinalM3}) and the definition of $\kappa(\rho)$ to derive an inequality on the $\beta$-norm of $S_T f$
\[ \begin{aligned}
\|S_T f\|_{\beta} &= \|S_T f\|_{L^1} + \beta \|S_T f\|_{\bar{\omega}_{n+1}} \\
&\leq \tilde{\eta} \|f\|_{L^1} + \beta \big( - C_3 T \|S_T f\|_{\bar{\omega}_{n}} + \|f\|_{\bar{\omega}_{n+1}} + \kappa(\rho) T \|f\|_{L^1} \big) \\
&\leq \beta \|f\|_{\bar{\omega}_{n+1}} + (\tilde{\eta}+ \kappa(\rho) T\beta ) \|f\|_{L^1} - \beta C_3 T \|S_T f\|_{\bar{\omega}_{n}}.
\end{aligned} \]
Finally, we choose $0 < \beta \leq \frac{1 - \tilde{\eta}}{ \kappa(\rho) T}$ and deduce
\begin{equation}\begin{aligned}
\label{EqPreliminaryBeta}
\|S_T f\|_{\beta} + C_3 \beta T \|S_T f\|_{\bar{\omega}_{n}} \leq \|f\|_{\beta}. 
\end{aligned}\end{equation} 

\vspace{.5cm}

\textbf{Step 4.} We prove that a slightly different version of (\ref{EqPreliminaryBeta}) also holds in the case (\ref{EqAlternative2}). 
From (\ref{IneqFinalM3}), using (\ref{EqAlternative2}), we have, for $T$, $\kappa(\rho)$ fixed as above
\[ \begin{aligned}
\|S_T f\|_{\bar{\omega}_{n+1}} + C_3 T \|S_T f\|_{\bar{\omega}_{n}} &\leq \|f\|_{\bar{\omega}_{n+1}} + \frac{\kappa(\rho) T}{A} \|f\|_{\bar{\omega}_{n}}.
\end{aligned} \]
Since $A \geq \frac{3 \kappa(\rho)}{C_3}$, it follows that
\[ \begin{aligned}
\|S_T f\|_{\bar{\omega}_{n+1}} + C_3 T \|S_T f\|_{\bar{\omega}_{n}} &\leq  \|f\|_{\bar{\omega}_{n+1}} + \frac{C_3T}{3} \|f\|_{\bar{\omega}_{n}}.
\end{aligned} \]
Using this inequality and the $L^1$ contraction we deduce
\begin{equation}\begin{aligned}
\label{EqPreliminaryBetaCase2}
\|S_T f\|_{\beta} + C_3 \beta T \|S_T f\|_{\bar{\omega}_{n}} &= \|S_T f\|_{L^1} + \beta \|S_T f\|_{\bar{\omega}_{n+1}} + C_3 \beta T \|S_T f\|_{\bar{\omega}_{n}}  \\
&\leq \|f\|_{L^1} + \beta \|f\|_{\bar{\omega}_{n+1}} + \beta \frac{C_3 T}{3} \|f\|_{\bar{\omega}_{n}}  \\
&= \|f\|_{\beta} + \beta C_3 \frac{T}{3} \|f\|_{\bar{\omega}_{n}}.
\end{aligned}\end{equation}

\vspace{.5cm}

\textbf{Step 5.} For $\beta$ as above and $\alpha = C_3 \beta T$, we have $\vertiii{.}_{\bar{\omega}_{n+1}} = \|.\|_{\beta} + \alpha \|.\|_{\bar{\omega}_n}$. Gathering (\ref{EqPreliminaryBeta}) and (\ref{EqPreliminaryBetaCase2}), we conclude that (\ref{EqContractionLemmaIntermediate}) holds and we deduce
\[ \begin{aligned}
\vertiii{S_T f}_{\bar{\omega}_{n+1}} \leq \vertiii{ f}_{\bar{\omega}_{n+1}}.
\end{aligned} \]

Since $\bar{\omega}_{n+1} \geq \bar{\omega}_{n} \geq 1$ on $G$, we conclude that for all $f \in L^1_{\bar{\omega}_{n+1}}(G)$ with $\langle f \rangle = 0$, 
\begin{equation}\begin{aligned}
\label{IneqM3Contraction}
\|S_T f\|_{\bar{\omega}_{n+1}} \leq M_{n+1} \|f\|_{\bar{\omega}_{n+1}},
\end{aligned}\end{equation}
for some constant $M_{n+1} \geq 1$.

\vspace{.3cm}

The proof of the second statement follows from similar arguments, note in particular that Step 1 can be adapted by using Lemma \ref{LemmaLyapunov} case (2) instead of case (1), and that the argument giving the existence of $\rho_0$ from the properties of $\bar{\omega}_n$ still applies and gives a new $\tilde{\rho}_0$ (hence a $\tilde{T}_0$ playing the role of $T_0$) when considering $\tilde{\omega}_n$ instead of $\bar{\omega}_n$. The remaining steps follow by straightforward adaptations. 
\end{proof}

%\begin{rmk}
%\label{RmkModifiedWeightOmega} For $n \in \{2,3\}$, for $(x,v) \in \bar{D} \times \RR^n$, set $\tilde{\omega}_i(x,v) = (\langle x,v\rangle)^{i - \frac12}$, $i \in \{1,\dots,4\}$. We can apply the previous steps with those weights. In particular, note that a similar argument to the one for $\omega_n$ holds for $\tilde{\omega}_n$ in Step 2, and that Lemma \ref{LemmaLyapunov} gives the Lyapunov inequalities corresponding to (\ref{EqLyapunovm2}). We conclude that there exists $\tilde{T}_0 > 0$ such that for all $T > \tilde{T}_0$, there exists $\tilde{M}_{n+1} > 0$ satisfying, for all $f \in L^1_{\tilde{\omega}_{n+1}}(G)$ with $\langle f \rangle = 0$,  
%$$ \|S_T f\|_{\tilde{\omega}_{n+1}} \leq \tilde{M}_{n+1} \|f\|_{\tilde{\omega}_{n+1}}. $$ 
%\end{rmk}

\subsection{Proof of Theorem \ref{ThmMain}}

\label{SubsectionMainProof}
In this subsection, we conclude the proof of Theorem \ref{ThmMain} using Lemma \ref{LemmaContractionm3}.
We consider the weights $w_1(x,v) = \langle x,v \rangle \ln( \langle x, v \rangle)^{0.1}$, and $w_0(x,v) = \ln( \langle x, v \rangle)^{0.1}$ for all $(x,v) \in \bar{D} \times \RR^n$. Recall the definition of the weights $\omega_i$ from (\ref{EqDefWeightsOmega}) for all $i \in \{1, \dots, 4\}$. We want to prove a decay rate for $S_t(f-g)$ with $f,g \in L^1_{\omega_{n+1}}$, $\langle f \rangle = \langle g \rangle$. We assume without loss of generality that $g \equiv 0$ so that $f \in L^1_{\omega_{n+1}}(G)$ with $\langle f \rangle = 0$.

\vspace{.5cm} 

\textbf{Step 1. } Recall that we write $L^1_{w,0}(G) = \{g \in L^1_w(G), \langle g \rangle = 0\}$, and the notation $M_1$ from Section \ref{SectionInterpolation}. We introduce the bounded projection $P: L^1(G) \to L^1_0(G)$  such that for all $h \in L^1(G)$ and $(x,v) \in G$,
\begin{equation}\begin{aligned}
\label{EqProjectionCorollary}
Ph(x,v) = h(x,v) - \frac{M_1(v)\|v\|^2} {c_1 |D|} \int_{G} h(y,w) dy dw,
\end{aligned}\end{equation} with $c_1 = \int_{\RR^n} M_1(v) \|v\|^2 dv < \infty$, where we recall that $|D|$ denotes the volume of $D$. One can see by a simple use of hyperspherical coordinates that $Ph \in L^1_{\omega_{n+1},0}(G)$ assuming $h \in L^1_{\omega_{n+1}}(G)$.  
Note that there exists $C > 0$ such that $\|Ph\|_{\omega_{n+1}} \leq C \|h\|_{\omega_{n+1}}$ for all $h \in L^1_{\omega_{n+1}}(G)$ and $\|Ph\|_{L^1} \leq C \|h\|_{L^1}$, and, since  $\langle h \rangle = 0$ implies $Ph = h$, $P$ is a bounded projection as claimed. Let $T > (T_0 \vee \tilde{T}_0)$ with $T_0$, $\tilde{T}_0$ given by Lemma \ref{LemmaContractionm3}. 
From Theorem \ref{ThmPositivityL1contraction}, we have
\[ \vertiii{S_T}_{L^1_0(G) \to L^1_0(G)} \leq 1,\]
and from Lemma \ref{LemmaContractionm3}, 
\[ \vertiii{S_T}_{L^1_{\tilde{\omega}_{n+1},0}(G) \to L^1_{\tilde{\omega}_{n+1},0}(G)} \leq \tilde{M}_{n+1}. \]
We apply Corollary \ref{firstCorollInterpol} with the projection $P$ and the values: 
\begin{enumerate}
\item $A_1 = \tilde{A}_1 = L^1(G)$, and, using the definition of $P$, $A'_1 = \tilde{A}'_1 = L^1_0(G)$,
\item $A_2 = \tilde{A}_2 = L^1_{\tilde{\omega}_{n+1}}(G)$, and, using the definition of $P$, $A'_2 = \tilde{A}'_2 = L^1_{\tilde{\omega}_{n+1}, 0}(G)$,
\item $\theta = \frac{2}{2n+1} \in (0,1)$ so that $A_{\theta} = \tilde{A}_{\theta} = L^1_{\mu}(G)$, where $\mu$ is defined on $\bar{D} \times \RR^n$ by \[\mu(x,v) = \langle x,v \rangle = \tilde{\omega}_{n+1}(x,v)^{\frac{2}{2n+1}},\] and, using the definition of $P$, $\tilde{A}'_{\theta} = A'_{\theta} = (A'_1 + A'_2) \cap A_{\theta} = L^1_{\mu,0}(G)$.
\end{enumerate}
We conclude that there exists $C_{\mu} > 0$ such that 
\[ \|S_T f\|_{\mu} \leq C_{\mu} \|f\|_{\mu}.\]

 Using Corollary \ref{Coroll2Goulaouic}, we obtain, for $\mu'(x,v) = \ln(\langle x, v \rangle)$ on $\bar{D} \times \RR^n$, since $f \in L^1_{\mu',0}(G)$, $\|S_T f\|_{\mu'} \leq C_{\mu'} \|f\|_{\mu'}$ for some constant $C_{\mu'} > 0$. Finally, since $w_0(x,v) = \mu'(x,v)^{0.1}$ for all $(x,v) \in \bar{D} \times \RR^n$, we apply one more time Corollary \ref{firstCorollInterpol}  with the projection $P$ to conclude that, for all $T > T_0 \vee \tilde{T}_0$, there exists $\tilde{W}_0 \geq 1$ such that, using $f \in L^1_{w_0,0}(G)$,
\[ \begin{aligned}
 \|S_T f\|_{w_0}\leq \tilde{W}_0 \|f\|_{w_0}.
 \end{aligned} \]
Since $(S_t)_{t \geq 0}$ is a strongly continuous semigroup of operators on $L^1_{w_0}(G)$, this implies, using the growth bound of the semigroup, that there exists $W_0 \geq 1$ such that for all $t \in (0,T)$, for all $f \in L^1_{w_0,0}(G)$,
 \begin{equation}\begin{aligned}
 \label{EqContractionW0}
  \|S_T f\|_{w_0} = \|S_{T-t} S_t f\|_{w_0} \leq W_0 \|S_t f\|_{w_0}. 
  \end{aligned}\end{equation}
 
 \vspace{.3cm}

\textbf{Step 2. } Using Lemma \ref{LemmaLyapunov}, case (3), and (\ref{EqContractionW0}), we obtain, for some constants $C, W_0 > 0$,
\[ \begin{aligned}
\|S_T f\|_{w_1} + \frac{T}{W_0} \|S_T f\|_{w_0} \leq \|f\|_{w_1} + C(1+T)\|f\|_{L^1},
\end{aligned} \] which rewrites
\[ \|S_T f\|_{w_1} + \frac{T}{W_0} \|S_T f\|_{w_0} \leq \|f\|_{w_1} + \kappa(\rho) T \|f\|_{L^1}, \]
with, for all $\rho > 0$, $\kappa(\rho) = \frac{C(1+T(\rho))}{T(\rho)}$ so that $\kappa \leq C'$ for some constant $C'$ independent of $\rho$. Set $w_0(r) = \ln(r)^{0.1}$ for $r \geq 1$. Since $\frac{w_0(\rho)}{\kappa(\rho)} \to \infty$ when $\rho \to \infty$, one can replicate the arguments of Steps 2 to 4 of the proof of Lemma \ref{LemmaContractionm3}. We obtain 
\begin{equation} \begin{aligned}
\label{IneqBetaAlpham1m0}
\|S_T f\|_{\beta} + 3 \alpha \|S_T f\|_{w_0} &\leq \|f\|_{\beta} + \alpha \|f\|_{w_0}
\end{aligned} \end{equation} just as (\ref{EqPreliminaryBetaCase2}), for $T = T(\rho)$ large enough with $T > T_0$, $T > \tilde{T}_0$ where $T_0, \tilde{T}_0$ are given by Lemma \ref{LemmaContractionm3}, with $\beta > 0$ constant, $\alpha = \frac{\beta T}{3W_0}$ and
\begin{equation} \begin{aligned}
\label{EqDefBetaNormW1}
\|f\|_{\beta} &:= \|f\|_{L^1} + \beta \|f\|_{w_1}.
\end{aligned} \end{equation}
\vspace{.5cm}

\textbf{Step 3. } We have, from our definition of $w_i$, $i \in \{0,1\}$ and of $\omega_{n+1}$, for $(x,v) \in G$,
\[ \begin{aligned}
w_1(x,v) &= \langle x, v \rangle \ln (\langle x, v \rangle)^{0.1} \\
& = \langle x, v \rangle \ln (\langle x, v \rangle)^{0.1} \mathbf{1}_{\{ \langle x, v \rangle < \lambda \}} + \langle x, v \rangle \ln (\langle x, v \rangle)^{0.1} \mathbf{1}_{\{ \langle x, v \rangle \geq \lambda \}} \\
&\leq w_0(x,v) \lambda + \frac{\langle x, v \rangle^{n+1} \ln (\langle x, v \rangle)^{-1.6} \ln (\langle x, v \rangle)^{1.7}}{\langle x, v \rangle^{n}} \mathbf{1}_{\{ \langle x, v \rangle \geq \lambda \}} \\
&\leq w_0(x,v) \lambda + \frac{\omega_{n+1}(x,v) \ln(\lambda)^{1.7}}{\lambda^{n}} \\
&\leq w_0(x,v) \lambda + \omega_{n+1}(x,v) \epsilon_{\lambda},
\end{aligned} \]
for $\lambda$ large enough, with $\epsilon_{\lambda} = \frac{\ln(\lambda)^{1.7}}{\lambda^{n}} \to 0$ as $\lambda \to \infty$, where we used that $x \to \frac{\ln(x)^{1.7}}{x^{n}}$ is non-increasing on $(e^2, \infty)$ and that $\langle x,v \rangle \geq e^2$ for all $(x,v) \in G$. 
We deduce, since $w_1(x,v) \geq 1$ for $(x,v) \in G$, 
\begin{equation}\begin{aligned}
\label{IneqControlLambdaBeta}
\frac{1}{\lambda(1+\beta)} \|f\|_{\beta} = \frac{1}{\lambda(1+\beta)}  (\|f\|_{L^1} + \beta \|f\|_{w_1}) \leq \frac{1}{\lambda} \|f\|_{w_1} \leq \|f\|_{w_0} + \frac{\epsilon_{\lambda}}{\lambda} \|f\|_{\omega_{n+1}}. 
\end{aligned}\end{equation}
Moreover, for $\tilde{\beta}$, $\tilde{\alpha}$ the positive values used to define $\vertiii{.}_{\omega_{n+1}}$ when applying Lemma \ref{LemmaContractionm3} with $\epsilon = 0.6$, one has, setting $B = \alpha/\tilde{\beta}$,
\begin{equation}\begin{aligned}
\label{IneqTripleNormOmega}
\frac{\alpha \epsilon_{\lambda}}{\lambda} \|S_T f\|_{\omega_{n+1}} = \frac{\alpha}{\tilde{\beta}} \frac{\epsilon_{\lambda}}{\lambda} \tilde{\beta} \|S_T f\|_{\omega_{n+1}} \leq B \frac{\epsilon_{\lambda}}{\lambda} \vertiii{S_T f}_{\omega_{n+1}},
\end{aligned}\end{equation}
%
%
% \frac{\beta T \epsilon_{\lambda}}{3 W_0 \lambda} \|S_T f\|_{\omega_{n+1}} + C_3 \frac{\beta T \epsilon_{\lambda}}{\lambda} \|S_T f\|_{\omega_n} + \frac{\epsilon_{\lambda}}{\lambda} \|S_T f\|_{L^1}, 
%\end{aligned} \]
%so that, setting $B = (1 \vee \frac{T}{3W_0}) \geq 1$,
%\begin{equation}\begin{aligned}
%
% \frac{\alpha \epsilon_{\lambda}}{\lambda} \|S_T f\|_{\omega_{n+1}} &\leq B \frac{\epsilon_{\lambda}}{\lambda} \big(\beta \|S_T f\|_{\omega_{n+1}} + C_3 \beta T \|S_T f\|_{\omega_n} + \|S_T f\|_{L^1} \big) \\
% & = B \frac{\epsilon_{\lambda}}{\lambda} \vertiii{S_T f}_{\omega_{n+1}}, 
% \end{aligned}\end{equation}
with the definition given in Lemma \ref{LemmaContractionm3} for $\vertiii{.}_{\omega_{n+1}}$.
Let $\delta := \frac{\alpha}{1+\beta}$, $Z := 1 + \frac{\delta}{\lambda}$ with $\lambda \geq \lambda_0 \geq 1$, $\lambda_0$ large enough so that $Z \leq 2$. Then
\[ \begin{aligned}
Z(\|S_T f\|_{\beta} + \alpha \|S_T f\|_{w_0}) &\leq \|S_T f\|_{\beta} + \frac{\alpha}{\lambda(1+\beta)}  \|S_T f\|_{\beta} + Z \alpha \|S_T f\|_{w_0} \\
& \leq \|S_T f\|_{\beta} + \alpha \|S_T f\|_{w_0} + \frac{\alpha \epsilon_{\lambda}}{\lambda} \|S_T f\|_{\omega_{n+1}} + Z \alpha \|S_T f\|_{w_0} \\
&\leq \|S_T f\|_{\beta} + 3 \alpha \|S_T f\|_{w_0} + \frac{B \epsilon_{\lambda}}{\lambda} \vertiii{S_T f}_{\omega_{n+1}} \\
&\leq \|f\|_{\beta} + \alpha \|f\|_{w_0} + \frac{B \epsilon_{\lambda}}{\lambda} \vertiii{S_T f}_{\omega_{n+1}},
\end{aligned} \]
where we used (\ref{IneqBetaAlpham1m0}), (\ref{IneqControlLambdaBeta}) and (\ref{IneqTripleNormOmega}). 
We introduce the norm $\vertiii{.}_{w_1}$ defined, for all function $h \in L^1_{w_1}(G)$, by \[\vertiii{h}_{w_1} := \|h\|_{\beta} + \alpha \|h\|_{w_0},\] so that the previous inequality rewrites
\[ Z \vertiii{S_T f}_{w_1} \leq \vertiii{f}_{w_1} + \frac{B \epsilon_{\lambda}}{\lambda} \vertiii{S_T f}_{\omega_{n+1}}. \]

\vspace{.5cm}

\textbf{Step 4.} We set $u_0 = \vertiii{f}_{w_1}$, and, for $k \geq 1$, $u_k = \vertiii{S_{kT} f}_{w_1}$. We also set  $v_0 = \vertiii{f}_{\omega_{n+1}}$ and, for $k \geq 1$, $v_k = \vertiii{S_{kT} f}_{\omega_{n+1}}$. Note that $v_k \leq v_0$ for all $k \geq 1$ by Lemma \ref{LemmaContractionm3}. Setting $Y = \frac{B \epsilon_{\lambda} }{\lambda}$, the previous inequality writes, 
\[ \begin{aligned}
Zu_1 \leq u_0 + Y v_1.
\end{aligned} \]
Iterating this inequality, we obtain
\[ \begin{aligned}
Z^k u_k \leq u_0 + Y \sum_{i=1}^k Z^{i-1} v_i,
\end{aligned} \]
from which we conclude that
\[ u_k \leq Z^{-k}u_0 + \frac{YZ}{Z-1} \sup_{i \geq 1} v_i \leq Z^{-k}u_0 + \frac{YZ}{Z-1} v_0 . \]
From this we deduce, recalling the definition of the $\beta$-norm (\ref{EqDefBetaNormW1}) and that   $Z \leq 2$,  $w_1 \leq \omega_{n+1}$
\[ \begin{aligned}
\vertiii{S_{kT}f}_{w_1} &\leq \frac{1}{(1+\delta/\lambda)^k} (1 + \beta + \alpha) \|f\|_{w_1} + \epsilon_{\lambda} \frac{2 B}{\delta} \vertiii{f}_{\omega_{n+1}} \\
&\leq C \Big(e^{-\frac{kT}{\lambda} \frac{\delta}{2T}} + \epsilon_{\lambda}\Big) \|f\|_{\omega_{n+1}},
\end{aligned} \]
with $C > 0$ explicit, where we used that $\vertiii{\cdot}_{\omega_{n+1}} \lesssim \|\cdot\|_{\omega_{n+1}}$. \bb We set $T_1 = kT$ and choose $\lambda = \Big(\frac{T_1 \frac{\delta}{2T}}{\ln \big(T_1^{n} \big)}\Big)$ with $k \geq k_0$, $k_0 \geq 1$ large enough so that $\lambda > \lambda_0$ and $T_1 > \exp(1)$ to obtain
\[ \begin{aligned}
\vertiii{S_{T_1}f}_{w_1} \leq C'(n) \Big(\frac{\ln(T_1)^{n+2}}{(T_1)^{n}} \Big) \|f\|_{\omega_{n+1}},
\end{aligned} \]
for some $C'(n) > 0$ depending explicitely on $C$, independent of $k$, where we used the trivial inequality $\frac{T_1}{\ln(T_1)} \leq T_1$. Upon modifying the definition of $C'(n)$ so that the previous inequality also holds for $k \in \{1,\dots k_0-1\}$, we can rewrite it as
\begin{equation}\begin{aligned}
\label{IneqRateResultw_1}
\vertiii{S_{kT}f}_{w_1} \leq C'(n) \Theta(k) \|f\|_{\omega_{n+1}},
\end{aligned}\end{equation}
with $\Theta(k) = \frac{\ln(kT)^{n+2}}{(kT)^{n}}$ for all $k \geq 1$.
\dd
\vspace{.5cm}

\textbf{Step 5.} With the norm $\vertiii{.}_{w_1}$, (\ref{IneqBetaAlpham1m0}) rewrites 
\[ \vertiii{S_T f}_{w_1} + 2 \alpha \|S_T f\|_{w_0} \leq \vertiii{f}_{w_1}. \]
By iterating this inequality and summing, we obtain, for $l \geq 1$, writing $[x]$ for the floor of $x \in \RR$,
\begin{equation}\begin{aligned}
\label{IneqPrelimFinalMain}
0 \leq \vertiii{S_{lT} f}_{w_1} + 2\alpha \sum_{k = [\frac{l}{2}] + 1}^l \|S_{kT} f\|_{w_0} \leq  \vertiii{S_{{[\frac{l}{2}]}T} f}_{w_1}.
\end{aligned}\end{equation}
Note that for any $1 \leq k \leq l$,
\[ \|S_{lT} f\|_{L^1} \leq \|S_{kT} f\|_{L^1} \leq \|S_{kT} f\|_{w_0}. \] Hence, from   (\ref{IneqRateResultw_1}) and (\ref{IneqPrelimFinalMain}), 
\[ \begin{aligned}
\min(1, 2\alpha) \Big(l - \Big[\frac{l}{2}\Big] + 1\Big) \|S_{lT} f\|_{L^1} \leq C'(n) \Theta\Big( \Big[\frac{l}{2} \Big] \Big) \|f\|_{\omega_{n+1}},
\end{aligned} \]
so that, for some $C > 0$
\[ \|S_{lT} f\|_{L^1} \leq C \frac{\ln(lT)^{n+2}}{(lT)^{n+1}} \|f\|_{\omega_{n+1}}. \]
We conclude to the desired rate by standard semigroup properties.

\subsection{Proof of Theorem \ref{ThmEquilibrium} and Corollary \ref{MainCorol}}
\label{SubsectionProofEquilibrium}

In this subsection, we prove Theorem \ref{ThmEquilibrium} and Corollary \ref{MainCorol} using the result of Theorem \ref{ThmMain}. We first show the following lemma, from which we will deduce both the uniqueness property in Theorem \ref{ThmEquilibrium} and Corollary \ref{MainCorol}. Recall the definition of $m_n$ from (\ref{EqDefWn}) and that $m_n \equiv \omega_{n+1}^{\frac{n}{n+1}}$ on $G$.

\begin{lemma}
\label{LemmaRatemn}
There exists an explicit constant $C' > 0$ such that for all $t \geq 0$, for all $f,g \in L^1_{\omega_{n}}(G)$ with $\langle f \rangle = \langle g \rangle$, there holds
\[ \|S_t (f-g)\|_{L^1} \leq \frac{C' \ln(1+t)^{n+1}}{(1+t)^{n}} \|f-g\|_{m_{n}}. \]
\end{lemma}

\begin{proof}
 We set $\tilde{f} := f - g$ so that $\langle \tilde{f} \rangle = 0$ and $\tilde{f} \in L^1_{m_n,0}(G)$.
From Theorem \ref{ThmPositivityL1contraction}, we have for all $t \geq 0$,  \[ \vertiii{S_t}_{L^1_0(G) \to L^1_0(G)} \leq 1,\] and from Theorem \ref{ThmMain},
\[\vertiii{S_t}_{L^1_{\omega_{n+1},0}(G) \to L^1_0(G)} \leq C \frac{\ln(1+t)^{n+2}}{t^{n+1}} = C \tilde{\Theta}(t), \]
the last equality standing for a definition of $\tilde{\Theta}(t)$, with $C > 0$ independent of $t$. 
We introduce the projection $P: L^1(G) \to L^1_0(G)$ given, for $h \in L^1(G)$ by
\[ Ph(x,v) = h(x,v) - \frac{M_1(v)\|v\|^2}{|D| c_1} \int_G h(y,w) dw dy, \quad (x,v) \in G, \]
with $c_1 = \int_{v \in \RR^n} M_1(v) \|v\|^2 dv$ a normalizing constant, see Section \ref{SectionInterpolation} for the definition of $M_1$.
Note that if $h \in L^1_{\omega_{n+1}}(G)$, $Ph \in L^1_{\omega_{n+1},0}(G)$ as one can check using hyperspherical coordinates, and that $\langle Ph \rangle = 0$. Moreover, $P$ sends $L^1_r(G)$ to $L^1_{r,0}(G)$ for any weight $1 \leq r \leq \omega_{n+1}$ and is bounded.

We apply Corollary \ref{firstCorollInterpol} with the projection $P$ and 
\begin{enumerate}[i.]
\item $A_1 = \tilde{A}_1 = \tilde{A}_2 = L^1(G)$, 
\item $A_2 = L^1_{\omega_{n+1}}(G)$, 
\item $A'_1 = \tilde{A}'_1 = \tilde{A}'_2 = L^1_0(G)$, $A_2' = L^1_{\omega_{n+1},0}(G)$ ,
\item $\theta = \frac{n}{n+1}$ so that $A_{\theta} = L^1_{m_n}(G)$, $\tilde{A}_{\theta} = L^1(G)$,
\item  $A'_{\theta} = (A'_1 + A'_2) \cap A_{\theta} = L^1_{m_n,0}(G)$ and  $\tilde{A}'_{\theta} = (\tilde{A}'_1 + \tilde{A}'_2) \cap \tilde{A}_{\theta} = L^1_0(G)$. 
\end{enumerate}
We deduce that for some explicit constant $C' > 0$, for all $t > 0$,
\[ \vertiii{S_t}_{L^1_{m_n,0}(G) \to L^1_0(G)} = C' \tilde{\Theta}(t)^{\frac{n}{n+1}} = C' \frac{\ln(1+t)^{\frac{n(n+2)}{n+1}}}{(1+t)^n} \leq C' \frac{\ln(1+t)^{n+1}}{(1+t)^n}. \] 

\vspace{-.5cm}
\end{proof}

\begin{proof}[Proof of Theorem \ref{ThmEquilibrium}]

\textbf{Step 1: Uniqueness. } Assume that there exists two functions $f_{\infty}, g_{\infty}$, both belonging to $L^1_{m_n}(G)$, with the desired properties.  Applying Lemma \ref{LemmaRatemn}, we have, for some $C > 0$, for all $t \geq 0$,
\[  \|S_t(f_{\infty} - g_{\infty})\|_{L^1} \leq C \frac{\ln(1+t)^{n+1}}{(t+1)^n} \|f_{\infty} - g_{\infty}\|_{m_n}. \]
For all $t \geq 0$, we have $S_t f_{\infty} = f_{\infty}$ and  $S_t g_{\infty} = g_{\infty}$. Set $\delta(t) := C \frac{\ln(1+t)^{n+1}}{(t+1)^n}$. We deduce that, for all $t \geq 0$,
\[ \|f_{\infty} - g_{\infty}\|_{L^1} \leq \delta(t) \|f_{\infty} - g_{\infty}\|_{m_n}. \]
We conclude that $f_{\infty} = g_{\infty}$ a.e. on $G$ since $\delta(t) \to 0$ as $t \to \infty$.

\vspace{.5cm}

\textbf{Step 2: Existence.} Recall that for all $k \in \llbracket n-1, n+1 \rrbracket$, for all $(x,v)$ in $D \times \RR^n$, $m_{k}(x,v) = \langle x, v \rangle^{k} \ln(\langle x, v \rangle)^{-1.6 \frac{n}{n+1}}$ . Let $g \in L^1_{m_{n+1}}(G)$ with $g \geq 0$ and $\langle g \rangle = 1$. We apply Lemma \ref{LemmaContractionm3} with $k = n+1$ and $\epsilon = 1.6 \frac{n}{n+1} - 1 \in (0,1)$, so that $\bar{\omega}_{n+1} = m_{n+1}$ and $\bar{\omega}_n = m_n$ and fix $T \geq T_0$. We set, for all $k \geq 1$,
\[ g_k := S_{Tk} g \qquad \text{and }f_k : = g_{k+1} - g_k.\] By mass conservation, for all $k \geq 1$, $\langle g_k \rangle = 1$ so that $\langle f_k \rangle = 0$ and $f_k \in L^1_{m_{n+1},0}(G)$. Applying (\ref{EqContractionLemmaIntermediate}), we have, for some constants $\beta, \alpha > 0$, setting $\|.\|_{\beta} = \|.\|_{L^1} + \beta \|.\|_{m_{n+1}}$, for all $k \geq 1$,
\[ \|S_T f_k\|_{\beta} + \alpha \|S_T f_k\|_{m_n} \leq \|f_k\|_{\beta} + \frac{\alpha}{3} \|f_k\|_{m_n}. \]
We introduce the modify norm $\vertiii{.}_{\tilde{\alpha}}$ defined by 
$\vertiii{.}_{\tilde{\alpha}} = \|.\|_{\beta} + \frac{\alpha}{3} \|.\|_{m_n}, $
so that the previous inequality reads
\begin{equation}\begin{aligned}
\label{EqModifiedAlphaTilde}
 \vertiii{S_T f_k}_{\tilde{\alpha}} + \frac{2 \alpha}{3} \|S_T f_k\|_{m_n} \leq \vertiii{f_k}_{\tilde{\alpha}}. 
 \end{aligned}\end{equation}
This implies that
\[ \vertiii{ f_{k+1}}_{\tilde{\alpha}} \leq \vertiii{f_k}_{\tilde{\alpha}}, \]
for all $k \geq 1$, so that the sequence $(\vertiii{f_k}_{\tilde{\alpha}})_{k \geq 1}$ is non-negative, non-decreasing, and is thus a converging subsequence. We fix $\epsilon > 0$. The previous observation implies that for $N \geq 0$ large enough, $p > l \geq N$, 
\[ 0 \leq \vertiii{f_l}_{\tilde{\alpha}} - \vertiii{f_p}_{\tilde{\alpha}} \leq \frac{2\alpha}{3} \epsilon. \]

Let $N$ as before, $p > l \geq N$. We have, using (\ref{EqModifiedAlphaTilde}) 
\[ \begin{aligned}
\frac{2 \alpha}{3} \Big\|g_{p+1} - g_{l+1}\Big\|_{m_n} &= \frac{2 \alpha}{3} \Big\|\sum_{k = l+1}^{p} f_k \Big\|_{m_n} \\
&\leq \sum_{k = l}^{p-1}  \frac{2 \alpha}{3 } \Big\|S_T f_k \Big\|_{m_n} \\
&\leq \sum_{k = l}^{p-1} \Big( \frac{2 \alpha}{3 }\|S_T f_k\|_{m_n} + \vertiii{S_T f_k}_{\tilde{\alpha}} \Big) - \sum_{k = l}^{p-1} \vertiii{S_T f_k}_{\tilde{\alpha}} \\
&\leq \sum_{k = l}^{p-1} \vertiii{f_k}_{\tilde{\alpha}} - \sum_{k = l}^{p-1} \vertiii{S_T f_k}_{\tilde{\alpha}} = \vertiii{f_l}_{\tilde{\alpha}} - \vertiii{f_p}_{\tilde{\alpha}} \leq \frac{2\alpha}{3} \epsilon,
\end{aligned} \]
by choice of $p$ and $l$. We deduce that the sequence $(g_k)_{k \geq 1}$ is a Cauchy sequence in the Banach space $L^1_{m_n}(G)$, hence converges towards a limit $f_{\infty} \in L^1_{m_n}(G)$ with $\langle f_{\infty} \rangle = \langle g \rangle = 1$ by mass conservation. A similar argument to the one in Step 1 can be used to prove that this limit is independent of the choice of $g \in L^1_{m_{n+1}}(G)$ with $\langle g \rangle = 1$. 
\end{proof}

\begin{proof}[Proof of Corollary \ref{MainCorol}]
We simply apply Lemma \ref{LemmaRatemn} with $g = f_{\infty}$, $f_{\infty}$ given by Theorem \ref{ThmEquilibrium}, to obtain Corollary \ref{MainCorol}.
\end{proof}

\color{black}

\section{Free-transport with absorbing boundary condition}
\label{SectionAbsorbing}

We consider in this section the free transport equation with absorbing condition at the boundary, which corresponds to (\ref{Problem2}).
We make the same assumptions as before on $D$, $n$ and $x \to n_x$. This problem is well-posed in the $L^1$ setting, since the boundary operator has norm $0$, see \cite[Theorem 3.5]{ArkerydCercignani1993}. Assuming $f_0 \in L^1(G)$, the characteristic method gives an explicit solution for all times $t \geq 0$, almost all $(x,v) \in G$,
\begin{equation}\begin{aligned}
\label{EqExplicitAbsorbing}
f(t,x,v) = f_0(x-tv,v) \mathbf{1}_{\{t < \sigma(x,-v)\}},
\end{aligned}\end{equation}
where $\sigma(x,v)$ is defined by (\ref{NotatSigma}) for all $(x,v)$ in $\bar{D} \times \RR^n$. This explicit formula makes the positivity of (\ref{Problem2}) clear. Obviously mass is not preserved by this problem. In what follows, we write $(S_t)_{t \geq 0}$ for the semigroup of linear operators corresponding to this evolution problem.  For $f_0 \in L^1(G)$, and $f$ the associated solution to (\ref{Problem2}) on $[0,\infty) \times G$, the trace $\gamma f(.,.,.)$ is well-defined on $[0, T) \times \pD \times \RR^n$ for any $T > 0$, see \cite[Theorem 1]{Mischler99}. Moreover, from \cite[Corollary 1]{Mischler99}, 
\begin{equation}\begin{aligned}
\label{EqCorollMischler2}
 |\gamma f(t,x,v)| = \gamma |f|(t,x,v) \quad \text{a.e. in } \big((0, \infty) \times \partial_+ G \big) \cup \big((0,\infty) \times \partial_- G\big).
 \end{aligned}\end{equation}
From the explicit solution (\ref{EqExplicitAbsorbing}), one easily deduces the convergence towards the equilibrium distribution given by $f_{\infty}(x,v) = 0$ for all $(x,v) \in \bar{D} \times \RR^n$.
We study the convergence rate of (\ref{Problem2}) towards equilibrium.
We recall that the weights $r_{\nu}$ for $\nu > 0$, are given by
\begin{equation}\begin{aligned}
r_{\nu}(x,v) = (1+ \sigma(x,v))^{\nu}, \qquad (x,v) \in \bar{D} \times \RR^n.
\end{aligned}\end{equation}

\begin{thm}
\label{ThmAbsorbing}
For any $f \in L^1_m(G)$, $t \geq 0$,
\[ \|S_t f\|_{L^1} \leq \Theta(t) \|f\|_{m}, \] with the following choices 
\begin{enumerate}[i.]
\item $m(x,v) = e^{\sigma(x,v)}$ in $\bar{D} \times G$, $\Theta(t) = e^{-t}.$
\item $m(x,v) = r_{\nu}(x,v)$ in $\bar{D} \times G$, $\Theta(t) = \frac{1}{(t+1)^{\nu}}$, $\nu > 1.$
\end{enumerate}
\end{thm}

\begin{proof}
Recall that $(v \cdot \nabla_x \sigma(x,v))= - 1$ for all $(x,v) \in G$, see (\ref{EqDerivativeSigma}). Note that, as a trivial consequence of the boundary condition, $\gamma S_t f = 0$ on $\partial_- G$ for all $f \in L^1(G)$.

For \textit{i.}, we have, by definition of $(S_t)_{t \geq 0}$, using also (\ref{EqCorollMischler2}), 
\[ \begin{aligned}
\frac{d}{dt} \int_{G} |S_t f| e^{\sigma(x,v)} dv dx &= \int_{G} (- v \cdot \nabla_x |S_t f|) e^{\sigma(x,v)} dvdx  \\
&= - \int_{G} |S_t f| e^{\sigma(x,v)} dv dx + \int_{\pD \times \RR^n} |\gamma(S_t) f| (v \cdot n_x) e^{\sigma(x,v)} dv d\zeta(x) \\
&= - \|S_t f\|_{m} + 0 - \int_{\partial_+ G} |\gamma(S_t f)| |v \cdot n_x| e^{\sigma(x,v)} dv d\zeta(x) \\
&\leq - \|S_t f\|_{m} 
\end{aligned} \]
where we used Green's formula (recall that $n_x$ is the unit inward normal vector at $x \in \pD$) and the boundary condition. We conclude with a straightforward application of Gr\"onwall's lemma.

\vspace{.5cm}

To prove \textit{ii.}, we differentiate the $L^1_{r_{\nu}}(G)$ norm of $S_t f$ and use the same arguments as in case \textit{i.} to obtain
\[ \begin{aligned}
\frac{d}{dt} \int_{G} |S_t f| (1 + \sigma(x,v))^{\nu} dv dx &= \int_{G} (-v \cdot \nabla_x |S_t f|) (1 + \sigma(x,v))^{\nu} dv dx \\
& = - \nu \int_{G} |S_t f| (1 + \sigma(x,v))^{\nu-1} dv dx \\
&\quad - \int_{\partial_+ G} |\gamma(S_t f)| |v \cdot n_x| r_{\nu}(x,v) dv d\zeta(x) \\
& \leq - \nu \int_{G} |S_t f| (1 + \sigma(x,v))^{\nu-1} dv dx.
\end{aligned} \]
Writing $\mathcal{B}$ for the generator of $(S_t)_{t \geq 0}$, we proved 
\begin{equation}\begin{aligned}
\label{IneqAdjointTransport}
\mathcal{B}^*r_{\nu} \leq - \phi(r_{\nu}),
\end{aligned}\end{equation} 
with for all $x \geq 1$, $\phi(x) = \nu x^{\frac{\nu-1}{\nu}}$, so that $\phi$ is strictly concave. 
We set for all $u \geq 1$, \[ H(u) = \int_1^u \frac{ds}{\phi(s)} = (u^{1/\nu} - 1),\]
\[ \text{ so that } H^{-1}(y) = (y + 1)^{\nu}, \quad \forall y \geq 0. \]
\[ \text{We also set } \forall t \geq 0, u \geq 1, \qquad \psi(t,u) := H^{-1}(H(u) + t) = (H(u) + t + 1)^{\nu}, \] which satisfies, for all $u \geq 1$, for all $t \geq 0$,
\begin{linenomath}
\begin{subequations}
\begin{alignat}{2}
\label{IneqDiffPsi1}
\partial_t \psi(t,u) &= \nu (H(u) + t + 1)^{\nu - 1}  = \phi(\psi(t,u)), \\
\label{IneqDiffPsi2}
\text{ and } \partial_u \psi(t,u) &= H'(u) \nu (H(u)+t+1)^{\nu - 1} = \frac{\phi(\psi(t,u))}{\phi(u)}. 
\end{alignat}
\end{subequations}
\end{linenomath}
We have, for all $t \geq 0$,
\[ \begin{aligned}
\frac{d}{dt} \int_G |S_t f| \psi(t,r_{\nu}) dv dx  &= \int_{G} \mathcal{B} |S_t f| \psi(t,r_{\nu}) + |S_t f| \partial_t \psi(t,r_{\nu}) dv dx \\
&= \int_{G}  |S_t f| \Big(\mathcal{B}^* \psi(t,r_{\nu}) + \partial_t \psi(t,r_{\nu}) \Big) dv dx \\
&= \int_{G}  |S_t f| \Big( (\mathcal{B}^* r_{\nu})  \partial_u \psi(t,r_{\nu}) + \partial_t \psi(t,r_{\nu}) \Big) dv dx \leq 0,
\end{aligned} \]
using (\ref{IneqAdjointTransport}) along with (\ref{IneqDiffPsi1}) and (\ref{IneqDiffPsi2}). Finally we use this inequality to conclude:
\[ \begin{aligned}
(t+1)^{\nu} \|S_t f\|_{L^1} \leq \int_G |S_t f| (H(r_{\nu}) + t + 1)^{\nu} dv dx = \|S_t f\|_{\psi(t,r_{\nu})} \leq \|f\|_{\psi(0,r_{\nu})} = \|f\|_{r_{\nu}}.
\end{aligned} \]

\vspace{-.7cm}

\end{proof}

%\begin{thm}
%Let $m(x,v) = e^{\sigma(x,v)}$ for all $(x,v) \in \bar{D} \times \RR^n$. There exists $C > 0$ such that for any $f \in L^1(m)$, $t \geq 0$,
%\begin{equation}\begin{aligned}
%\label{EqDissipativityAbsorbtion}
%\|S_t f\|_{L^1(m)} \leq C e^{-t} \|f\|_{L^1(m)}.
%\end{aligned}\end{equation}
%\end{thm}
%
%\begin{proof}
%Recall that $(v \cdot \nabla_x \sigma(x,v))= - 1$ for all $(x,v) \in D \times \RR^n$, see (\ref{EqDerivativeSigma}). We have, by definition of $(S_t)_{t \geq 0}$,
%\[ \begin{aligned}
%\frac{d}{dt} \int_{D \times \RR^n} |S_t f| e^{\sigma(x,v)} dx dv &= \int_{D \times \RR^n} (- v \cdot \nabla_x |S_t f|) e^{\sigma(x,v)} dx dv \\
%&= - \int_{D \times \RR^n} |S_t f| e^{\sigma(x,v)} dx dv = - \|S_t f\|_{m}.
%\end{aligned} \]
%\end{proof}

\begin{rmk}
The assumption $f \in L^1_{e^{\sigma(x,v)}}(G)$ is very restrictive for $L^1(G)$ functions positive in a neighborhood of $\{\|v\| = 0\}$. On the other hand, if $f \in L^1_{e^{\sigma(x,v)}}(G)$, the problem is indeed dissipative. For such an initial datum (\ref{Problem1}) shows only the convergence rate of Theorem \ref{ThmMain}, of polynomial order. This emphasizes the fact that the convergence rate in Theorem \ref{ThmMain} depends both on the regularity (inverse moments) of the initial data, and on the boundary condition, which becomes the limiting factor for the rate of convergence with very regular initial data. 
\end{rmk}

\bibliographystyle{alpha}
\bibliography{biblio}

\begin{thebibliography}{EGKM13}
\bibitem[AC93]{ArkerydCercignani1993}
L.~Arkeryd and C.~Cercignani.
\newblock A global existence theorem for the initial-boundary-value problem for
  the {B}oltzmann equation when the boundaries are not isothermal.
\newblock {\em Archive for Rational Mechanics and Analysis}, 125(3):271--287,
  Sep 1993.
\end{thebibliography}

\end{document}